\documentclass[12pt,reqno]{amsart}

\usepackage{amsmath,amssymb,amsthm,amscd,amsfonts,amsbsy,bm}
\usepackage[english]{babel}
\usepackage[mathscr]{eucal}

\usepackage{graphicx}

\usepackage{xcolor}

\usepackage[active]{srcltx}

\numberwithin{equation}{section}

\renewcommand{\a}{\alpha}
\renewcommand{\b}{\beta}
\newcommand{\g}{\gamma}
\newcommand{\G}{\Gamma}
\renewcommand{\d}{\delta}
\newcommand{\D}{\Delta}
\newcommand{\e}{\epsilon}

\newcommand{\z}{\zeta}

\renewcommand{\th}{\theta}

\renewcommand{\k}{\kappa}

\newcommand{\m}{\mu}

\newcommand{\x}{\xi}

\renewcommand{\r}{\rho}
\newcommand{\s}{\sigma}

\renewcommand{\t}{\tau}

\newcommand{\F}{\Phi}
\newcommand{\h}{\chi}

\renewcommand{\O}{\Omega}

\newcommand{\vs}{\varsigma}

\newcommand{\T}{{\mathbb T}}
\newcommand{\C}{{\mathbb C}}
\newcommand{\R}{{\mathbb R}}
\newcommand{\Z}{{\mathbb Z}}

\newcommand{\DD}{{\mathbb D}}

\newcommand{\ab}{{\mathbf a}}

\newcommand{\kb}{{\mathbf k}}

\newcommand{\rb}{{\mathbf r}}

\newcommand{\vb}{{\mathbf v}}

\newcommand{\Bb}{{\mathbf B}}

\newcommand{\Db}{{\mathbf D}}

\newcommand{\Pb}{{\mathbf P}}

\newcommand{\MF}{\mathfrak M}

\newcommand{\Ac}{{\mathcal A}}

\newcommand{\Ec}{{\mathcal E}}

\newcommand{\Hc}{{\mathcal H}}

\newcommand{\Kc}{{\mathcal K}}

\newcommand{\Pc}{{\mathcal P}}

\newcommand{\SO}{{\rm SO}\,}

\newcommand{\supp}{\hbox{{\rm supp}}\,}

\newcommand{\dbar}{\overline{\partial}}

\newcommand{\bg}{\pmb{\gamma}}

\newtheorem{theorem}{Theorem}[section]
\newtheorem{proposition}[theorem]{Proposition}

\theoremstyle{definition}
\newtheorem{definition}[theorem]{Definition}

\theoremstyle{remark}
\newtheorem{remark}[theorem]{Remark}
\newtheorem{example}[theorem]{Example}

\begin{document}

\title[Bergman space Toeplitz operators]{Toeplitz operators defined by sesquilinear forms: Bergman space case}

\author[Rozenblum]{Grigori Rozenblum}
\address{Department of Mathematics, Chalmers  University of Technology,   Gothenburg, Sweden}
\email{grigori@chalmers.se}

\author[Vasilevski]{Nikolai Vasilevski }
\address{Department of Mathematics, Cinvestav, Mexico-city, Mexico}
\email{nvasilev@duke.math.cinvestav.mx}

\begin{abstract}The definition of Toeplitz operators in the Bergman space $\Ac^2(\mathbb{D})$ of square integrable analytic functions in the unit disk in the complex plane is extended in such way that it covers many cases where the traditional definition does not work. This includes, in particular, highly singular symbols such as measures, distributions, and certain hyper-functions.

\end{abstract}

\maketitle

\section{Introduction}\label{intro}
Toeplitz operators play an important role in Mathematics and Applications. Having appeared almost hundred years ago, as matrices with elements being constant along  diagonals parallel to the main one, these operators have undergone several stages of generalization, as more applications  required this. One of their most common (but not the most general) definitions is as follows.

Let $\Hc$ be a Hilbert space of functions on a set $\Omega\subset\R^d$ or $\Omega\subset\C^d$, and let $\Ac$ be its closed subspace, with $\Pb:\Hc\to\Ac$ being the orthogonal projection. For a function $a$ on $\Omega$ (usually measurable and bounded), the Toeplitz operator on $\Ac$ with symbol $a$ is defined as $T_a: \Ac\ni f\mapsto \Pb(af)\in\Ac$.

Probably, the best studied class of Toeplitz operators consists of  Riesz-Toeplitz ones, where the space $\Hc$ is $L_2(\T)$, $\T\subset\C$ being the unit circle, and $\Ac$ is the Hardy space consisting of boundary values of analytic functions. A lot of attention have been attracted recently to another class of Toeplitz operators, for which the enveloping space $\Hc$ is $L_2(\O)$, $\O$ being the unit disk in $\C$ (the unit ball or the polydisk in $\C^d$), alternatively  the whole $\C^d$ or $\R^d$, and the subspace $\Ac$ consists of solutions of some elliptic equation or system (the leading examples here are the spaces of analytic or harmonic functions.) These latter Toeplitz operators are called the Bergman-Toeplitz ones. The present paper is a continuation of \cite{RozVas}, and is devoted to the development of the theory of the Bergman-Toeplitz operators acting on the Bergman space of analytic functions on the unit disk $\mathbb{D}$.

Toeplitz operators, as they were initially defined, make immediate sense for a bounded symbol $a$. In fact, the multiplication by a bounded function is a bounded operator in $L_2$, so the operator $T_a$ is automatically bounded as a composition of two bounded operators. For Riesz-Toeplitz operators it seems to be unlikely to extend essentially their definition to more general classes of symbols. On the other hand, for Bergman-Toeplitz operators, it turns out to be possible to define, in a reasonable way, the operators whose symbols belong to certain classes of unbounded functions, measures, distributions, and even some more general objects. These generalizations, which prove to be quite useful, stem from the special property of the spaces $\Ac$: they are \emph{reproducing kernel Hilbert spaces}, the ones, for which each evaluation  functional at a point of $\O$ is bounded with respect to the $L_2$-norm.

A far going generalization of the traditional approach to Toeplitz operators has been developed in \cite{RozVas}, where  the notion of the \emph{sesquilinear form symbol} was introduced. The concrete study was concentrated there around operators in the Fock space, the space of entire analytic function in $\C$, square  integrable against the Gaussian measure. It turned out that quite a number of reasonable operators fail to be Toeplitz in frame of the traditional definition, while they fit conveniently into the sesquilinear forms setting. This, in particular, enabled us to introduce even the operators with symbol being  hyper-functions of a special kind.

In the present paper we follow the pattern of \cite{RozVas} and consider Toeplitz operator in  another leading case, the Bergman space of analytic function in the unit disk, square integrable with respect to the Lebesgue measure. Having defined the general sesquilinear forms as symbols, we present a number of particular classes of symbols, including the ones absent in the Fock space case. This study includes, in particular, a convenient sesquilinear forms representation for operators in commutative subalgebras, where the structure of the sesquilinear form is closely related to the structure of the algebra. We also present a number of examples of operators that fail to be Toeplitz in the traditional setting but fit quite well into our more general picture. In some cases we give necessary and sufficient conditions for a subspace to be the range of a Toeplitz operator with symbol from the prescribed class.

The final part of the paper is devoted to the study of operators with highly singular symbols, the ones involving hyper-functions or, in other words, distributions of unbounded differential order. By means of certain new estimates for the $k$-Carleson measures, we find sufficient conditions for such distributions to define bounded Toeplitz operators.

 One should notice that under our approach, the role of the enveloping Hilbert space $\Hc$ is not that important, and, actually, can be eliminated altogether. The latter circumstance becomes crucial when considering Toeplitz operators in  spaces of solutions of elliptic equations or systems in the whole space,  where there may be no natural candidate for the role of the enveloping space. We will consider such situations in publications to follow.

\section{Preliminaries}
\label{se:general}

We recall here the general approach \cite{RozVas} to the extension of the notion of Bergman-Toeplitz (just Toeplitz in what follows) operators for a wide class of singular symbols, given in the context of reproducing kernel Hilbert spaces.

Let $\mathcal{H}$ be a Hilbert space of functions defined in a domain $\Omega$ in $\mathbb{R}^d$ or $\mathbb{C}^d$, with $\langle\cdot,\cdot\rangle$, $\|\cdot\|$ denoting its inner product, respectively, norm, and let $\mathcal{A}$ be a closed subspace of $\mathcal{H}$ on which every evaluation functional $\Ac\ni f\mapsto f(z) $, $z\in\Omega,$ is bounded. For each $z \in \Omega$, we denote by  $\kb_z(\cdot) \in \mathcal{A}$  the element realizing by the Riesz theorem the evaluation functional at the point $z$, i.e.,
\begin{equation*}
  f(z) = \langle f(\cdot), \kb_z(\cdot) \rangle, \ \ \ \ \mathrm{for \ all} \ \ \ f \in \mathcal{A};
\end{equation*}
in this case the orthogonal projection $\Pb$ of $\mathcal{H}$ onto $\mathcal{A}$ has the form
\begin{equation*}
  \Pb \, : \mathcal{H}\ni f(z) \ \longmapsto \ \langle f(\cdot), \kb_z(\cdot) \rangle \in \mathcal{A},
\end{equation*}
or
\begin{equation*} 
 (\Pb f)(z) = \langle f(\cdot), \kb_z(\cdot) \rangle.
\end{equation*}

Recall that a bounded sesquilinear form $F(\cdot,\cdot)$ on $\Ac$ is a mapping
\begin{equation*}
 F(\cdot,\cdot) \, : \ \Ac \oplus \Ac \longrightarrow \ \mathbb{C},
\end{equation*}
linear in the first argument and anti-linear in the second one,  satisfying the condition
\begin{equation*}
 |F(f,g)| \leq C \|f\|\cdot \|g\|, \ \ \ \ \mathrm{for \ all} \ \ \ f,\, g \in \Ac
\end{equation*}
with some constant $C>0$.
We adopt then the following vocabulary.\\
Given a  bounded sesquilinear form $F(\cdot,\cdot)$ on $\Ac$, we define the Toeplitz operator $T_F$ \emph{with sesquilinear form symbol} $F$ acting on $\Ac$ as follows
\begin{equation}\label{Def:operatorTF}
 (T_F f)(z) = F(f(\cdot),\kb_z(\cdot))
\end{equation}
If the sesquilinear form $F$ stems in some well defined way from some analytic object denoted, say, by $a$, the operator \eqref{Def:operatorTF} can be alternatively denoted by $T_a$, in order to stress this dependence, as long as this does not cause a misunderstanding.

Concrete operator considerations of thus defined Toeplitz operators were elaborated in \cite{RozVas}   for  operators acting in the classical Fock (or Segal-Bargmann) space. In the present paper we specify our approach to the operators acting in the Bergman space $\mathcal{A}^2(\mathbb{D})$ over the unit disk $\mathbb{D}$ in $\mathbb{C}$. This Bergman space was being studied for a long time, from different points of view, with numerous applications; we refer, in particular, to the books \cite{HeKoZ}, \cite{Zhu}.

That is, in what follows, the enveloping Hilbert space is $\Hc=L_2(\DD)$ with respect to the \emph{normalized} Lebesgue measure $dV=\pi^{-1}dxdy$, $z=x+iy$, and  the subspace $\mathcal{A}= \mathcal{A}^2(\mathbb{D})$ is its closed subspace consisting of analytic functions. It is well known that $\mathcal{A}^2(\mathbb{D})$ is a reproducing kernel Hilbert space with
\begin{equation}\label{kernel}
  \kb_z(w)  =\overline{\kb_w(z)} = (1-\overline{z}w)^{-2},\,
\end{equation}
so that the orthogonal \emph{Bergman projection} $\Bb$ of $L_2(\DD)$ onto $\mathcal{A}^2(\mathbb{D})$ has the form
\begin{equation*}
 (\Bb f)(z) = \langle f(\cdot), \kb_z(\cdot) \rangle.
\end{equation*}

Recall that, given $a(z) \in  L_{\infty}(\DD)$, the classical Toeplitz operator with symbol $a$ is the compression onto $\mathcal{A}^2(\mathbb{D})$ of the multiplication operator $(M_a f)(z) = a(z)f(z)$ on $L_2(\DD)$:
\begin{equation}\label{Oper.naive}
 T_a \, : \ f \in \mathcal{A}^2(\mathbb{D}) \ \longmapsto \ \Bb(af) \in \mathcal{A}^2(\mathbb{D}).
\end{equation}
It is well known that in this case, the (function) symbol $a$ is uniquely, as an element in $L_\infty,$ defined by the Toeplitz operator \eqref{Oper.naive}.

Classical Toeplitz operators with bounded measurable symbols fit in our context as follows.

\begin{example}\textsf{Classical Toeplitz operators in the Bergman space} \\
We start with any \emph{bounded linear functional} $\Phi \in (L_1(\mathbb{D}))'$. Recall that such a functional is uniquely defined by a function $a \in L_{\infty}(\mathbb{D})$ and has the form
\begin{equation*}
 \Phi(u) \equiv \Phi_a(u) = \int_{\mathbb{D}} a(z)u(z)dV(z).
\end{equation*}
Note that for arbitrary $f,\, g \in \mathcal{A}^2(\mathbb{D})$ the product $f\overline{g}$ belongs to $L_1(\mathbb{D})$, and the finite linear combinations of such products are dense in $L_1(\mathbb{D})$. We define the sesquilinear forms $F_a$ as
\begin{equation} \label{eq:F_a}
 F_a(f,g) = \Phi_a(f\overline{g}).
\end{equation}
This form is obviously bounded:
\begin{equation*}
 |F_a(f,g)| \leq \|\Phi_a\| \|f\overline{g}\|_{L_1} \leq \|\Phi_a\| \|f\| \|g\|=\|a\|_{L_\infty}\|f\| \|g\|.
\end{equation*}
Then
\begin{eqnarray*}
 (T_{F_a} f)(z) &=& \Phi_a(f\overline{\kb_z}) = \int_{\mathbb{D}} a(w)f(w)\overline{\kb_z(w)}dV(w)\\
 &=&  \int_{\mathbb{D}} \frac{a(w)f(w)}{(1-z\overline{w})^2}\, dV(w) = (T_af)(z),
\end{eqnarray*}
i.e, the Toeplitz operator $T_{F_a}$ with sesquilinear form symbol $F_a$ coincides with the classical Toeplitz operator with symbol $a$. Moreover, all classical Toeplitz operators can be obtained starting with a functional $\Phi$ in $(L_1(\mathbb{D}))'$, which defines in turn  the form (\ref{eq:F_a}).
\end{example}

The approach of the above example immediately enables us to extend considerably the notion of Toeplitz operators for a wide class of defining sesquilinear form symbols.
In what follows we will use even more general approach based on the construction introduced in \cite{RozVas}, which is suitable for any reproducing kernel Hilbert space (for more details, results, and examples, see  \cite{RozVas}).

Let $X$ be a (complex) linear topological space (not necessarily complete). We denote by $X'$
its dual space (the set of all continuous linear
functionals on $X$), and denote by $\Phi(\phi) \equiv (\Phi,\phi)$ the intrinsic pairing of $\Phi
\in X'$ and $\phi \in X$. Let then $\mathcal{A}$
be a reproducing kernel Hilbert space with $\kb_z(\cdot)$ being its reproducing kernel
function.

By a continuous $X$-valued sesquilinear form $G$  on $\mathcal{A}$ we mean a continuous
mapping
\begin{equation*}
 G(\cdot,\cdot) \, : \ \mathcal{A} \oplus \mathcal{A}
 \longrightarrow \ X,
\end{equation*}
which is linear in the first argument and anti-linear in the second argument.

Then, given a continuous $X$-valued sesquilinear form $G$ and an element $\Phi \in X'$, we
define the sesquilinear form $F_{G,\Phi}$ on
$\mathcal{A}$ by

\begin{equation*} 
 F_{G,\Phi}(f,g) = \Phi(G(f,g)) = (\Phi, G(f,g)).
\end{equation*}
Being continuous, this form is bounded, and thus, by the Riesz theorem, defines a bounded (Toeplitz) operator $T_{G,\Phi}$ such that

\begin{equation}\label{eq:T_F(G,Psi)0}\langle T_{G,\Phi} f,g\rangle =(\Phi, G(f,g)).
\end{equation}

In particular, for $ g=\kb_z$, \eqref{eq:T_F(G,Psi)0} takes the form
\begin{equation*}
\langle T_{G,\Phi} f,\kb_z\rangle=(\Phi, G(f,\kb_z)).
\end{equation*}
Since $T_{G,\Phi} f\in \Ac$, the reproducing property of $\kb_z$ implies the explicit description of the action of the operator $T_{G,\Phi}$:

\begin{equation} \label{eq:T_F(G,Psi)}
 (T_{G,\Phi}f)(z) := (T_{F_{G,\Phi}}f)(z) = (\Phi, G(f,\kb_z)).
\end{equation}

We stress here that the sesquilinear form that defines a Toeplitz operator may have several quite different analytic expressions (involving different spaces $X$ and functionals $\Phi$), producing, nevertheless, the same Toeplitz operator. We recall as well that under the above approach the set of Toeplitz operators {\rm (\ref{eq:T_F(G,Psi)})} forms a $^*$-algebra, see the detailed reasoning in \cite{RozVas}.

\section{Measures and distributions as symbols}

In this section we describe how our general approach enables us to define in a reasonable way the Toeplitz operators with rather singular symbols.

\begin{example}\textsf{Toeplitz operators defined by Carleson measures} \\
Recall (see, for example, \cite[Section 7.2]{Zhu}) that a finite positive Borel measure $\mu$ on $\mathbb{D}$ is called a \emph{Carleson measure} (C-measure) for the Bergman space $\mathcal{A}^2(\mathbb{D})$ if there exists a constant $C > 0$ such that
\begin{equation} \label{eq:Carleson}
 \int_{\mathbb{D}} |f(z)|^2 d\mu(z) \leq C \int_{\mathbb{D}} |f(z)|^2 dV(z), \ \ \ \ \mathrm{for \ all} \ \  f \in \mathcal{A}^2(\mathbb{D}).
\end{equation}
It is well known (say, from the Cauchy representation formula) that a finite measure $\m$ with compact support strictly  inside $\DD$ is automatically a C-measure, so the property of a given measure $\m$ to be a C-measure is determined only by its behavior near the boundary of $\DD$, see, e.g., \cite{Zhu}.

Given a C-measure $\m$, we set $X = L_1(\mathbb{D}, d\mu)$, $X' = L_{\infty}(\mathbb{D}, d\mu)$, $G(f,g)(z) =
 f(z)\overline{g(z)}$, $\Phi= \Phi_1 = 1
 \in L_{\infty}(\mathbb{D}, d\mu)$, so that
\begin{equation*}
 F_{G,\Phi_1}(f,g) = \int_{\mathbb{D}} f(z)\overline{g(z)}d\mu := F_{\mu}(f,g).
\end{equation*}
This form is obviously bounded in $\Ac^2(\mathbb{D})$ by the Cauchy-Schwartz inequality and (\ref{eq:Carleson}).
Using \eqref{eq:T_F(G,Psi)}, we obtain the formula for the action of the resulting Toeplitz operator:
\begin{equation*}
 (T_{F_{\mu}} f)(z) = \int_{\mathbb{D}} f(w)\overline{\kb_z(w)}d\mu(w) = \int_{\mathbb{D}} \frac{f(w)d\mu(w)}{(1-z\overline{w})^2} = (T_{\mu}f)(z),
\end{equation*}
i.e, the Toeplitz operator $T_{F_\mu}$ with sesquilinear form symbol $F_{\mu}$ is nothing but the Toeplitz operator defined by a C-measure $\mu$ (see, for example, \cite[Section 7.1]{Zhu}).

A natural generalization of this situation involves a complex valued Borel measure $\mu$
such that its variation $|\mu|$ is a C-measure. In
such a case $X = L_1(\mathbb{D}, d\mu) := L_1(\mathbb{D}, d|\mu|)$, $X' = L_{\infty}(\mathbb{D},
d\mu) := L_{\infty}(\mathbb{D}, d|\mu|)$ with the same
formulas for $G(f,g)$ and $\Phi$ as before. \\
In particular, this description extends to the case of a positive C-measure $\mu$,
$X = L_1(\mathbb{D}, d\mu)$, $X' = L_{\infty}(\mathbb{D}, d\mu)$, $G(f,g)(z) =
f(z)\overline{g(z)}$, $\Phi= \Phi_a = a \in
L_{\infty}(\mathbb{D}, d\mu)$, so that
\begin{equation*}
 F_{G,\Phi_a}(f,g) = \int_{\mathbb{D}} a(z) f(z)\overline{g(z)}\,d\mu.
\end{equation*}
and
\begin{equation*}
 (T_{F_{G,\Phi_a}}f)(z) =  \int_{\mathbb{D}} \frac{a(w)f(w)\,d\mu(w)}{(1-z\overline{w})^2}.
\end{equation*}
\end{example}

\begin{example}\textsf{Toeplitz operators with distributional symbols} \\
Any distribution $\Phi$ in $\Ec'(\DD)$ has a finite order, and thus can be extended to a continuous functional in the space of functions with finite smoothness,
\begin{equation*}
    |\Phi(h)|\le C(\Phi)\|h\|_{C^N(\mathbb{K})},
\end{equation*}
for some $N$ and some compact set $\mathbb{K}\Subset \DD$ containing the support of $\Phi$.
The Cauchy formula implies  in a standard way that the $C^N(\mathbb{K})$-norm of the product $h=f\overline{g}$, $f,g\in \mathcal{A}^2(\mathbb{D})$, is majorated by the product of $\mathcal{A}^2(\mathbb{D})$-norms of $f$ and $g$. Therefore the sesquilinear form
\begin{equation} \label{eq:dist_Phi}
 F_{\Phi}(f,g) = \Phi(f\overline{g})
\end{equation}
is bounded on $\mathcal{A}^2(\mathbb{D})$ and thus defines a bounded operator.

The operator
\begin{equation*}
 T_{F_\Phi}:f(z) \mapsto \Phi\left(f(\cdot)\overline{\kb_z(\cdot)}\right)
\end{equation*}
defined in this manner coincides with the standardly defined \cite{AlexRoz,RozToepl, PeTaVi3} Toeplitz operator $T_{\Phi}$ with distributional symbol $\Phi$.

We give now a more explicit description of such operators. Let $\Phi$ be a distribution with compact support in $\DD$. By the structure theorem for such distributions (see, e.g., Theorem in Sect.4.4. in \cite{Gelf}, page 116-117 of the AP edition of 1968),  $\Phi$ admits a representation as a \emph{finite} sum
\begin{equation}\label{Struct}
    \Phi=\sum_{q}D^q u_q,
\end{equation}
where  $q=(q_1,q_2)$, $D=(D_1,D_2)$; $u_q$ are continuous functions which can be chosen as having support in an arbitrarily close neighborhood of the support of $\Phi$. Rearranging the derivatives, we can write \eqref{Struct} as the finite sum
\begin{equation}\label{Struct2}
    \Phi=\sum_{q}\partial^{q_1}\dbar^{q_2}u_q,
\end{equation}
again with (some other) continuous functions $u_q$, having support in an arbitrarily close neighborhood of $\supp \Phi$. Using \eqref{Struct2}, formula \eqref{eq:dist_Phi} is transformed to
\begin{equation*}
    (T_{\Phi}f)(z)=\sum_q (-1)^{q_1+q_2}\int_{\DD} u_q(w)\cdot \dbar_w^{q_2}\overline{\kb_w(z)}\, \partial ^{q_1}f(w)\,dV(w).
\end{equation*}

We present here several examples of the action of such operators $T_{\Phi}
$.
\begin{example} \label{ex:deriv_delta}
Let $\Phi$ be a derivative of the $\delta$-distribution,
$\Phi=\partial^\a\dbar^\b\delta$, so that $\Phi(g) = (-1)^{\a+\b}\partial^\a\dbar^\b g(0)$.
Then, by \eqref{eq:dist_Phi},
\begin{gather}\label{TFEx1}
    (T_{\Phi}f)(z) = (\partial^\a\dbar^\b\delta)(f(\cdot)k_z(\cdot)) = (-1)^{\a+\b}\, \dbar_w^\b \overline{k_z(0)}\,\partial^\a f(0)\\ \nonumber = (-1)^{\a+\b}(\b+1)!z^{\b}\partial^\a f(0).
\end{gather}
Considering finite linear combinations of the distributions of the form $\Phi$, we see that for any distribution having support at the origin, the corresponding Toeplitz operator with this distribution as symbol is well defined via the sesquilinear forms. Moreover, the range of such operator is contained in the space of polynomials.
\end{example}

\begin{example}More generally, for the derivative of the $\delta$--distribution at the point $\z\in\DD$, $\Phi=\partial^\a\dbar^\b\delta_\z$, we have, similar to \eqref{TFEx1},
\begin{equation}\label{TFEx2}
    (T_{\Phi}f)(z) = (-1)^{\a+\b}\, \frac{(\b+1)!z^{\b}}{(1-z\overline{\z})^{-(2+\b)}}\,\partial^\a f(\z).
\end{equation}
Analogously to Example \ref{ex:deriv_delta}, we see here that the range of the Toeplitz operators \eqref{TFEx2} is contained in the space of rational functions without simple poles.
\end{example}

\bigskip
It is convenient to normalize the symbol $\Phi$ of Example \ref{ex:deriv_delta}. Given $p,q \in \mathbb{Z}_+$, we consider the distributional symbol
\begin{equation*}
 \Phi_{p,q} = \frac{(-1)^{p+q}}{\sqrt{(p+1)(q+1)}\,p!q!}\, \partial^p\,\dbar^q\delta.
\end{equation*}
Then, by (\ref{TFEx1}), the Toeplitz operator $T_{\Phi_{p,q}}$ acts on the standard monomial basis of $\mathcal{A}^2(\mathbb{D})$
\begin{equation}\label{basis}
 e_k(z) = \sqrt{k+1}\, z^k, \ \ \ \ k \in \mathbb{Z}_+,
\end{equation}
as follows
\begin{equation*}
 T_{\Phi_{p,q}} e_k =
\left\{
\begin{array}{ll}
 e_q, & \mathrm{if} \ k = p, \\
 0, & \mathrm{otherwise}.
\end{array} \right.
\end{equation*}
Thus, the Toeplitz operator $T_{\Phi_{p,q}}$  is the rank one operator $P_{p,q}f = \langle f, e_p \rangle e_q$ on $\mathcal{A}^2(\mathbb{D})$. In the special case $p=q$ the operator $T_{\Phi_{p,p}}$ is the one-dimensional orthogonal projection onto the subspace generated by $e_p$. It is important to note that  $T_{\Phi_{p,p}}=P_{p,p}$ is the diagonal operator having the eigenvalue sequence $(0,..., 0, 1, 0,0,...)$ with respect to the basis \eqref{basis}.
Further on, we abbreviate $\Phi_{p,p}=\Phi_p$ and $P_{p,p}=P_p$.

We mention as well that the linear combinations of the rank one Toeplitz operators $T_{\Phi_{p,q}} = P_{p,q}$, $p,q \in \mathbb{Z}_+$, form a dense subset both in the space of all finite rank and in the space of all compact operators on $\mathcal{A}^2(\mathbb{D})$ in the norm operator topology. To see this, it is sufficient to observe that any rank one operator $P_{u,v} = \langle \cdot, u \rangle v$, where $u,\, v  \in \mathcal{A}^2(\mathbb{D})$, can be
be norm approximated by finite rank operators of the forms
\begin{equation*}
  \sum_{p=1}^{p_0} \sum_{q=1}^{q_0}\overline{u_p}\,v_q P_{p,q},
\end{equation*}
where $u_p = \langle u, e_p\rangle$ and $v_q = \langle v, e_q\rangle$ are the Fourier coefficients of $u$ and $v$ respectively.

Note that, when extending the class of symbols from $L_{\infty}$-functions to  more general objects we may lose the property of uniqueness of a symbol for a given Toeplitz operator; the same Toeplitz operator can have now quite different defining symbols.

Indeed, consider the following singular integral operators
\begin{equation*}
 (S_{\mathbb{D}}\varphi)(z) = - \int_{\mathbb{D}} \frac{\varphi(\zeta)}{(\zeta - z)^2} dV(\zeta)  \ \ \mathrm{and} \ \  (S^*_{\mathbb{D}}\varphi)(z) = - \int_{\mathbb{D}} \frac{\varphi(\zeta)}{(\overline{\zeta} - \overline{z})^2} dV(\zeta),
\end{equation*}
which are known to be bounded on $L_2(\mathbb{D})$ and mutually adjoint.  It was established in \cite{SanVas}, that for all $p,q \in \mathbb{Z}_+$,
\begin{equation} \label{eq:SIO-symbol}
 T_{\Phi_{p,q}} =
 \left\{\begin{array}{ll}
   T_{(S_{\mathbb{D}}^*)^{q}(S_{\mathbb{D}})^{q} -(S_{\mathbb{D}}^*)^{q+1}(S_{\mathbb{D}})^{q+1} } T_{\sqrt{\frac{p+1}{q+1}}\overline{z}^{p-q}},
& p > q, \\
   T_{(S_{\mathbb{D}}^*)^{p}(S_{\mathbb{D}})^{p} -(S_{\mathbb{D}}^*)^{p+1}(S_{\mathbb{D}})^{p+1} },
& p=q, \\
   T_{\sqrt{\frac{q+1}{p+1}}z^{q-p}} T_{(S_{\mathbb{D}}^*)^{p}(S_{\mathbb{D}})^{p} -(S_{\mathbb{D}}^*)^{p+1}(S_{\mathbb{D}})^{p+1}} , & p < q.
\end{array}
\right.
\end{equation}

It is interesting to observe also the following phenomenon.  Consider a finite rank Toeplitz operator, which, by \cite{AlexRoz}, is a Toeplitz operator whose symbol is a finite combination of $\delta$-distributions and their derivatives. The support of its distributional symbol is a finite set of points located somewhere in the unit disk $\mathbb{D}$.
By the above arguments, this finite rank Toeplitz operator can be norm-approximated by finite rank Toeplitz operators whose distributional symbols have just one and the same one-point support $\{0\}$. Later on, in Sect. \ref{hyperfun}, this observation will be extended in more detail,
\end{example}

\section{Diagonalizable operators related to commutative algebras}

As  well known, any normal operator in the Hilbert space is diagonalizable with respect to its spectral decomposition. Given a family of (not necessarily normal) operators, one should not expect that they can be diagonalizable simultaneously. However, as it concerns Toeplitz operators on $\Ac^2(\mathbb{D})$, many diagonalizable  families, related to commutative subalgebras, have been found, see \cite{Vasil}. There are three model cases of them, and all others are obtained from these model cases via M\"obius transformations. In this section we show that the operators related to these model cases can be naturally viewed as our more generally defined Toeplitz operators, we give also the informative representations of the corresponding sesquilinear forms.
\subsection{Examples}
Toeplitz operators with radial symbols form one of the most well studied classes of Toeplitz operators.
  We start here with an example of a \emph{radial} operator that fails to be a Toeplitz one in the traditional setting.
\begin{example}\textsf{The reflection operator} \\
Let $(J\varphi)(z) = \varphi(-z)$ be the reflection operator in $\mathcal{A}^2(\mathbb{D})$. This operator is obviously bounded, and it acts on the standard monomial basis \eqref{basis} of $\mathcal{A}^2(\mathbb{D})$
as follows
\begin{equation*}
 (Je_k)(z) = (-1)^k e_k(z).
\end{equation*}
This means that $J$ is a diagonal operator with respect to the standard basis (\ref{basis}), and its  eigenvalue sequence has the form $\bg_J = \{(-1)^k\}_{k \in \mathbb{Z}_+}$. Thus (see \cite{Zorboska03}) the operator $J$ is  \emph{radial}. Assume now that $J$ is a Toeplitz operator, $J=T_a$ for some symbol $a \in L_{\infty}(\mathbb{D})$. Then (see \cite{Zhou_Chen_Dong_2011}) the symbol $a$ must be a \emph{radial} function $a=a(|z|)$, and by \cite{GrMaxVas}, the eigenvalue sequence $\bg_a$ of the operator $J=T_a$ should belong to the class $\SO(\mathbb{Z}_+)$ of slowly oscillating sequences introduced by Schmidt \cite{Schmidt_1924}. Recall that $\SO(\mathbb{Z}_+)$ consists of all $\ell_{\infty}$-sequences $\bg = \{\gamma(n)\}_{n \in \mathbb{Z}_+}$ satisfying the condition
\begin{equation*}
 \lim_{\frac{n}{m} \to 1} |\gamma(n) - \gamma(m)| = 0.
\end{equation*}
This is not the case for the operator under consideration, and therefore the operator $J$ cannot be a Toeplitz one with an $L_{\infty}$-function serving as symbol. By the same reason the operator $J$ cannot even belong to the algebra generated by Toeplitz operators with bounded measurable radial symbols. Moreover, $J$  is effectively separated from this algebra: as it is easy to show, for any operator $T$ in this
algebra,  the norm  $\|J - T \| $ is at least $1$, so the zero operator is the best norm approximation of $J$ by operators in the above algebra.

At the same time,we can set $X = L_1(\mathbb{D})$, $X' = L_{\infty}(\mathbb{D})$, $G(f,g) = f(-z)\overline{g(z)}$, $\Phi = 1 \in L_{\infty}(\mathbb{D})$,
so that
\begin{equation*}
 F_{G,\Phi}(f,g) = \int_{\mathbb{D}} f(-z)\overline{g(z)}dV(z),
 \end{equation*}
 and, therefore, $J = T_{G,\Phi}$.
\end{example}

We note that another form, defining $J$ as a Toeplitz operator, is given by (\ref{eq:radial-form}) with $\gamma(n) = (-1)^n$, $n \in \mathbb{Z}_+$.

\medskip
We give now an example of a (radial) bounded operator that belongs to the algebra generated by Toeplitz operators with bounded measurable (radial) symbols and which itself cannot be represented as a Toeplitz operator with such symbol.

\begin{example}\textsf{The rank one projection} \\
Consider the orthogonal projection $P_0 f = \langle f, e_0 \rangle e_0$ of $\mathcal{A}^2(\mathbb{D})$ onto the one-dimensional subspace generated by $e_0$. It is a diagonal, and thus \emph{radial}, operator having the eigenvalue sequence $\bg_{P_0} = (1,0,0,...)$. Since $\bg_{P_0} \in c_0 \subset \SO(\mathbb{Z}_+)$, the operator $P_0$ belongs (see \cite{GrMaxVas}) to the algebra generated by Toeplitz operators with bounded measurable radial symbols. But by \cite[Theorem 6.1.4]{Vasil} it can not be a Toeplitz operator with this kind of symbol.
At the same time $P_0$  presents a simple example of  a Toeplitz operator with a \emph{not uniquely defined} more general symbol.

The first representation of $P_0$ is as follows. Consider the distributional symbol $F_0 = \delta$. Then for the Toeplitz operator $T_{F_0}$ we have
\begin{equation*}
 (T_{F_0}e_k,e_l)=  e_k(0)\overline{e_l(0)}= \left\{
\begin{array}{ll}
 1, & k = l = 0, \\
 0, & \mathrm{otherwise}.
\end{array} \right.
\end{equation*}
This means that the Toeplitz operator $T_{F_0}$ is nothing but the above one-dimensional projection $P_0$.

The second representation of $P_0$,
\begin{equation*}
 P_0 = T_{I - S_{\mathbb{D}}^*S_{\mathbb{D}}},
\end{equation*}
follows from (\ref{eq:SIO-symbol}) when $p = q = 1$.

We mention as well that, although the operator $P_0$ cannot be itself represented as a Toeplitz operator with a bounded measurable radial symbol, it can be norm-approximated  by Toeplitz operators with such symbols:
\begin{equation*}
 P_0 = \lim_{n \to \infty} T_{a_n},
\end{equation*}
where $a_n(r) = (n+3)(1-r^2)^{n+2}$.

Indeed, the eigenvalue sequence  $\pmb{\g}_{a_n}$ of the operator $T_{a_n}$ has the form \cite[Theorem 6.1.1]{Vasil}
\begin{eqnarray*}
 \gamma_{a_n}(k) &=& (k+1)\int_0^1 a_n(\sqrt{r})r^kdr \\
 &=& (n+3)(k+1)\int_0^1 (1-r)^{n+2}r^kdr \\
 &=& (n+3)(k+1) B(n+3,k+1) = \frac {(n+3)!(k+1)!}{(n+k+3)!}.
\end{eqnarray*}
Then
\begin{equation*}
 \|T_{a_n} - P_0\| = \|\pmb{\gamma}_{a_n} - \pmb{\gamma}_{P_0} \|_{\ell_{\infty}} = \gamma_{a_n}(1) = \frac{2}{n+4},
\end{equation*}
which implies the desired: $P_0 = \lim_{n \to \infty} T_{a_n}$.
\end{example}

\subsection{Sesquilinear form symbols for operators in commutative  subalgebras}

In the previous subsection we considered two particular cases of the so-called \emph{radial} operators, the ones that are diagonal with respect to the standard monomial basis (\ref{basis}) in $\mathcal{A}^2(\mathbb{D})$.
Other important special  classes of diagonalizable operators are the so-called \emph{vertical} and \emph{angular} operators. These three classes appear naturally under the study of the (only!) three model cases of commutative $C^*$-algebras generated by Toeplitz operators with specific classes of $L_{\infty}$-symbols.
These three classes are (see \cite{Vasil} for details): the \emph{elliptic} case - Toeplitz operators on the disk $\mathbb{D}$ with symbols depending only on the modulus $r =|z|$ of $z = |z|e^{\imath\theta} \in \mathbb{D}$,  the \emph{radial} symbols; the \emph{parabolic} case - Toeplitz operators on the upper-half plane $\Pi$ with symbols depending only on the imaginary part $y$ of $z = x+\imath y \in \Pi$, the \emph{vertical} symbols; and the \emph{hyperbolic} case - Toeplitz operators on the upper half-plane $\Pi$ with symbols depending only on the polar angle $\theta$  of $z = |z|e^{\imath\theta} \in \Pi$, the \emph{angular} symbols.

We explore now these three diagonalizable classes and show that the corresponding operators are the Toeplitz ones defined by the appropriate sesquilinear forms, different for each case.

\begin{example}\textsf{Radial operators} \\
Recall \cite{Zorboska03}, that the radial operators $S$ (acting on $\mathcal{A}^2(\mathbb{D})$) are those that commute with the rotation operators $(U_t f )(z) = f(e^{-\imath t} z)$, $t \in \mathbb{R}$. They are diagonal with respect to the basis (\ref{basis}) in $\mathcal{A}^2(\mathbb{D})$:
\begin{equation*}
 \langle S e_n(z), e_m(z) \rangle = \delta_{n,m}\,\gamma_S(n),
\end{equation*}
and a Toeplitz operator $T_a$ with bounded symbol is radial if and only if its symbol $a$ is radial.

The spectral sequence $\bg_S = \{\gamma_S(n)\}_{n \in \mathbb{Z}_+}$ of the radial operator $S$ belongs to $\ell_{\infty}$, and the correspondence  $ S \longmapsto \bg_S$ gives  an isometric isomorphism between the $C^*$-algebra of radial operators and $\ell_{\infty}$.
The spectral sequence $\bg_{T_a} = \bg_a = \{\gamma_a(n)\}_{n \in \mathbb{Z}_+}$ of a Toeplitz operator $T_a$ with radial symbol is calculated by the formula \cite[Theorem 6.1.1]{Vasil}
\begin{equation*}
 \gamma_a(n) = (n+1) \int_0^1 a\left(\sqrt{r}\right) r^n dr,
\end{equation*}
and the $C^*$-algebra generated by Toeplitz operators with bounded radial symbols is isomorphic to the algebra $\mathrm{SO}(\mathbb{Z}_+)$, see \cite{GrMaxVas}.

Thus, in particular, if the spectral sequence $\bg_S$ of some radial operator $S$ does not belong to $\mathrm{SO}(\mathbb{Z}_+)$, then $S$ cannot be a Toeplitz operator  with bounded symbol. At the same time all bounded radial operators can be viewed as Toeplitz ones under the following construction.

The operator (see \cite[Corollary 10.3.4]{Vasil})
\begin{equation*}
  R_\rb\, : f(z) \ \longmapsto \left\{\int_{\mathbb{D}} f(z)\overline{e_n(z)}\, dV(z) \right\}_{n \in \mathbb{Z}_+}
\end{equation*}
maps isometrically  $\mathcal{A}^2(\mathbb{D})$ onto $\ell_2$. We introduce then the space   $X=\ell_1$, with $X' = \ell_{\infty}$, and the $\ell_1$-valued sesquilinear form on $\mathcal{A}^2(\mathbb{D})$:
\begin{equation*}
 G(f,g) = (R_\rb f)\,(\overline{R_\rb g}) = \{(R_\rb f)(n)\,(\overline{R_\rb g})(n)\}_{n \in \mathbb{Z}_+}.
\end{equation*}
Having any element $\Phi = \bg = \{\gamma(n)\}_{n \in \mathbb{Z}_+} \in \ell_{\infty} = X'$, we define the sesquilinear form
\begin{equation} \label{eq:radial-form}
 F_{G,\Phi}(f,g) = \sum_{n \in \mathbb{Z}_+} \gamma(n)\cdot (R_\rb f)(n)\,(\overline{R_\rb g})(n),
\end{equation}
which in turn defines the Toeplitz operator
\begin{equation*}
 (T_{G,\Phi}f)(z) = F_{G,\Phi}(f,\kb_z).
\end{equation*}
It is straightforward that each radial operator $S$ is of the form $S =
T_{G,\Phi_S}$, where $\Phi_S = \bg_S$. Classical Toeplitz operators $T_a$ with bounded radial symbols $a$ are exactly those for which $\Phi_{T_a} = \bg_a$.

Note that the use of an arbitrary element $\bg = \{\gamma(n)\}_{n \in \mathbb{Z}_+} \in \ell_{\infty} = X'$ in the form (\ref{eq:radial-form}) is much more natural than taking just elements $\bg_a$ from a dense subset of $\mathrm{SO}(\mathbb{Z}_+)$, covering classical radial Toeplitz operators.
\end{example}

The situation with two other classes is quite similar, and their consideration uses the same approach as in the radial case. Therefore, we list  the main facts only. In this discussion it is more natural to pass from the unit disk $\mathbb{D}$ to the upper half-plane $\Pi$. This can be done, for example, via the unitary operator
\begin{equation*}
 (U_0 f)(w) = \frac{2}{(1-iw)^2}\, f\left(\frac{w-i}{1-iw}\right),\, w\in\Pi,
\end{equation*}
which maps isometrically both $L_2(\mathbb{D})$ onto $L_2(\Pi)$, and $\mathcal{A}^2(\mathbb{D})$  onto $\mathcal{A}^2(\Pi)$.

\begin{example}\textsf{Vertical operators} \\
Recall (see \cite{HerreraYanezMaximenkoVasilevski2013}) that the vertical operators $S$ (acting on $\mathcal{A}^2(\Pi)$) are those that commute with the horizontal translation operators $(H_h f )(z) = f(z-h)$, $h \in \mathbb{R}$.

Introduce the operator (see \cite[Corollary 10.3.8]{Vasil})
\begin{equation*}
  (R_\vb f)(x) = \sqrt{x} \int_{\Pi} f(w)e^{-i\overline{w}x} dV(w),
\end{equation*}
which maps isometrically $\mathcal{A}^2(\Pi)$ onto $L_2(\mathbb{R}_+)$.

Then, by \cite[Theorem 2.3]{HerreraYanezMaximenkoVasilevski2013}, a bounded operator $S$ on $\mathcal{A}^2(\Pi)$ is vertical if and only if there exists a function $\bg \in L_{\infty}(\mathbb{R}_+)$ such that the operator $S$ is unitary equivalent to the multiplication operator by (its spectral) function $\bg$: $S = R_\vb^* M_{\bg} R_\vb$, where $M_{\bg}f = \bg f$.

The correspondence $ S \longmapsto \bg_S$, with $\bg_S$ being the spectral function of $S$, is an isometric isomorphism between the $C^*$-algebra of vertical operators and $L_{\infty}(\mathbb{R}_+)$.

A Toeplitz operator $T_a$ is vertical if and only if its symbol is vertical, $a = a(\mathrm{Im}\,w)$. The spectral function $\bg_{T_a} = \bg_a(x)$ of a vertical Toeplitz operator $T_a$ is calculated by the formula \cite[Theorem 5.2.1]{Vasil}
\begin{equation*}
 \bg_a(x) = 2x \int_{\mathbb{R}_+} a(t)e^{-2tx} dt, \qquad x \in \mathbb{R}_+.
\end{equation*}
The $C^*$-algebra generated by vertical Toeplitz operators with bounded symbols is isomorphic \cite{HerreraYanezMaximenkoVasilevski2013} to the algebra $\mathrm{VSO}(\mathbb{R}_+)$, which consists of those functions in $L_{\infty}(\mathbb{R}_+)$ that are uniformly continuous with respect to the logarithmic
metric
\begin{equation*}
 \rho(x,y) = \left| \ln x - \ln y \right|.
\end{equation*}
 Note also that the spectral functions of vertical Toeplitz operators themselves form a dense subset in  $\mathrm{VSO}(\mathbb{R}_+)$.

In order to fit such operators to our general setting, we introduce  $X=L_1(\mathbb{R}_+)$, with $X' = L_{\infty}(\mathbb{R}_+)$, and the following $L_1(\mathbb{R}_+)$-valued sesquilinear form on $\mathcal{A}^2(\Pi)$
\begin{equation*}
 G(f,g) = (R_\vb f)\,(\overline{R_\vb g}).
\end{equation*}
Then for any element $\Phi = \bg(x) \in L_{\infty}(\mathbb{R}_+) = X'$, we define
\begin{equation*}
 F_{G,\Phi}(f,g) = \int_{\mathbb{R}_+} \bg(x)\, (R_\vb f)(x)\,(\overline{R_\vb g})(x)\, dx,
\end{equation*}
and construct the corresponding Toeplitz operator
\begin{equation*}
 (T_{G,\Phi}f)(z) = F_{G,\Phi}(f,\kb_z).
\end{equation*}
It is straightforward that each vertical operator $S$ is of the form $S =
T_{G,\Phi_S}$, where $\Phi_S = \bg_S$. Classical Toeplitz operators $T_a$ with bounded vertical symbols $a$ correspond to the case $\Phi_{T_a} = \bg_a$.
\end{example}

\begin{example}\textsf{Angular operators} \\
Recall \cite{EsmMaxVas}, that the angular operators $S$ (acting on $\mathcal{A}^2(\Pi)$) are those that commute with the group of dilation operators $(\D_h f )(z) = hf(hz)$, $h \in \mathbb{R}_+$.

The operator (see \cite[Corollary 10.3.12]{Vasil})
\begin{equation*}
  (R_\ab f)(x) = \sqrt{\frac{x}{1 - e^{1-2\pi x}}} \int_{\Pi} \overline{w}^{-ix-1} f(w)dV(w),
\end{equation*}
maps isometrically $\mathcal{A}^2(\Pi)$ onto $L_2(\mathbb{R})$.

Then (\cite[Theorem 2.1]{EsmMaxVas}) a bounded operator $S$ on $\mathcal{A}^2(\Pi)$ is angular if and only if there exists a function $\bg \in L_{\infty}(\mathbb{R})$ such that the operator $S$ is unitary equivalent to the multiplication operator by its spectral function $\bg$: $S = R_\ab^* M_{\bg} R_\ab$.

The correspondence $ S \longmapsto \bg_S$ is an isometric isomorphism between the $C^*$-algebra of angular operators and $L_{\infty}(\mathbb{R})$.

A Toeplitz operator $T_a$ is angular if and only if its symbol is angular, $a = a(\mathrm{arg}\,w)$. The spectral function $\bg_{T_a} = \bg_a(x)$ of an angular Toeplitz operator $T_a$ is calculated by the formula \cite[Theorem 7.2.1]{Vasil}
\begin{equation*}
 \bg_a(x) = \frac{2x}{1-e^{-2\pi x}} \int_0^{\pi} a(\theta)e^{-2x\theta}\,d\theta, \qquad x \in \mathbb{R}.
\end{equation*}
The $C^*$-algebra generated by angular Toeplitz operators with bounded symbols is isomorphic \cite{EsmMaxVas} to the algebra $\mathrm{VSO}(\mathbb{R})$, which consists of those functions in $L_{\infty}(\mathbb{R})$ that are uniformly continuous with respect to the arcsinh-metric
\begin{equation*}
 \rho(x,y) = \left| \mathrm{arcsinh}\, x - \mathrm{arcsinh}\, y \right|.
\end{equation*}
Again the spectral functions of angular Toeplitz operators themselves form a dense subset in  $\mathrm{VSO}(\mathbb{R})$.

Introduce now $X=L_1(\mathbb{R})$, with $X' = L_{\infty}(\mathbb{R})$, and the following $L_1(\mathbb{R})$-valued sesquilinear form on $\mathcal{A}^2(\Pi)$
\begin{equation*}
 G(f,g) = (R_\ab f)\,(\overline{R_\ab g}).
\end{equation*}
For any element $\Phi = \bg(x) \in L_{\infty}(\mathbb{R}) = X'$, we define
\begin{equation*}
 F_{G,\Phi}(f,g) = \int_{\mathbb{R}} \bg(x)\, (R_\ab f)(x)\,(\overline{R_\ab g})(x)\,dx,
\end{equation*}
and construct the corresponding Toeplitz operator
\begin{equation*}
 (T_{G,\Phi}f)(z) = F_{G,\Phi}(f,\kb_z).
\end{equation*}
Again each angular operator $S$ is of the form $S = T_{G,\Phi_S}$, where $\Phi_S = \bg_S$. Classical Toeplitz operators $T_a$ with bounded angular symbols $a$ correspond to the case $\Phi_{T_a} = \bg_a$.
\end{example}

\section{Distributional symbols in $\Ec'(\DD)$ and finite rank operators}
\subsection{Finite rank Toeplitz operators}

One more source of examples of bounded operators that are not Toeplitz ones in their traditional meaning is provided by the finite rank theorems.

First of all, by D. Luecking \cite{Lue2}, neither nonzero \emph{finite rank operator} can be a Toeplitz operator with function symbol.

Passing to measure symbols,  we consider the linear subspace $\Kc$ of $\mathcal{A}^2(\mathbb{D})$ which consists of finite linear combinations of functions $\kb_w(\cdot)$, $w\in\DD$. This subspace is known to be dense in $\mathcal{A}^2(\mathbb{D})$. As it follows from the explicit formula \eqref{kernel} for the Bergman kernel,
any finite linear combination of $\kb_w(\cdot)$ is a rational function, having a finite limit at infinity and possessing exclusively  second order poles.
In particular, this means that $\Kc$ does not fill the whole  $\mathcal{A}^2(\mathbb{D})$.

\begin{example}\label{ExFinRankFnction} 
Consider the following finite rank operator on $\mathcal{A}^2(\mathbb{D})$
\begin{equation} \label{eq:fin-rank}
 (Tf)(z)=\sum_{j=1}^N \,\langle f,f_j\rangle\, g_j(z),
\end{equation}
where $f_j\in\mathcal{A}^2(\mathbb{D})$, $g_j\in\mathcal{A}^2(\mathbb{D})$, $j=1,2,...,N$, with at least one of $g_j$ outside $\Kc$. We suppose that the functions $f_j$ are linearly independent, as well as the functions $g_j$, so that the number of summands in \eqref{eq:fin-rank} cannot be reduced. This operator is \emph{not} a Toeplitz operator with measure-symbol. Indeed, suppose that there is a measure $\mu$ such that $T=T_\mu$. By Luecking's theorem in \cite{Lue2}, this measure must be a finite linear combination of point measures at no more points $w_j$ than the rank of the operator is, $\mu=\sum_{j=1}^N a_j\d_{w_j}$. Thus, we would have
\begin{equation*}
    (Tf)(z)=(T_\m f)(z)=\int_{\DD} f(w)\kb_w(z)\,d\m(w)=\sum_{j=1}^N a_j\kb_{w_j}(z)f(w_j).
\end{equation*}
This means that the range of $T$ is spanned by the functions $\kb_{w_j}(z)$, i.e., lies in $\Kc$. This contradicts the choice of $T$, made in such way that the range of $T$ is not contained in $\Kc$. Thus $T$ is \emph{not} a Toeplitz operator with measure-symbol.
\end{example}

A minor modification of the above reasoning produces the complete description of bounded operators on $\mathcal{A}^2(\mathbb{D})$ that are not Toeplitz operators with symbols-distributions having a compact support in $\mathbb{D}$.

Denote by $\mathcal{M}$ the linear subspace of functions that are finite linear combinations of the reproducing kernels $\kb_w(z)$ and their derivatives in $\overline{w}$ variable, $\kb^{l}_w(z)=\dbar^l_w\kb_w(z)$, for different  $w\in\DD$. If $w\ne 0$, such function $\kb^{l}_w$ is a rational one, with pole  of order $2+l$, lying at the point $\bar{w}^{-1}$ outside $\overline{\mathbb{D}}$. For $w=0$, the function $\kb^{l}_w(z)$ equals to $(l+1)!z^l$. Thus, $\mathcal{M}$ consists of rational functions having poles only outside $\overline{\mathbb{D}}$ and with all these poles being of the order  at least  two; therefore, $\mathcal{M}$ is dense in $\mathcal{A}^2(\mathbb{D})$ but does not cover the whole $\mathcal{A}^2(\mathbb{D})$.
In particular, the restriction to $\DD$ of any non-polynomial entire analytic function, or any rational function possessing at least one simple pole does not belong to $\mathcal{M}$.

\begin{example}\label{ExFinRankDistr} Let $T$ be a finite rank operator on $\mathcal{A}^2(\mathbb{D})$, $(Tf)(z)=\sum_{j=1}^N \langle f,f_j\rangle\,g_j(z)$, again, with both sets $\{f_j\}$, $\{g_j\}$ linearly independent and at least one of $g_j$ being not a rational function without simple poles. This operator is not a Toeplitz one, with a compactly supported distribution as a symbol. In fact, if this operator would be Toeplitz then, by Theorem 3.1 in \cite{AlexRoz}, its symbol has to be a finite linear combination of the $\d$-distributions and their derivatives at some finite set of points. Therefore, similarly to the measure case considered in Example \ref{ExFinRankFnction}, the range of the operator should belong to the space $\mathcal{M}$. As explained above, this contradicts to the choice of functions $g_j$ defining the finite rank operator.
\end{example}

Note also that a simple interpolation shows that any finite-dimensional subspace in $\mathcal{M}$ can serve as the range of some finite rank Toeplitz operator with symbol in $\Ec'(\DD)$.

\medskip

At the same time we show now that each finite rank operator is a Toeplitz operator in our extended sense, i.e., a Toeplitz operator defined by a sesquilinear form.

\begin{example}\textsf{Finite rank operators via sesquilinear forms} \\
We start with a finite rank operator (\ref{eq:fin-rank}) and set $X = \mathbb{C}^N$, $X' =
(\mathbb{C}^N)'$, identified with $\C^N$ by means of the
standard pairing. Then, let
 $G(f,g) = \left(\langle f, f_1 \rangle \langle g_1, g \rangle, ... , \langle f, f_N
 \rangle \langle g_N, g \rangle \right)$, $\Phi=$
 $  \Phi_{\mathbf{1}} = (1,...,1) \in \mathbb{C}^N$, so that
\begin{equation*}
 F_{G,\Phi_\mathbf{1}}(f,g) = \sum_{j=1}^N \langle f, f_j \rangle \langle g_j, g
 \rangle.
\end{equation*}
Another representation of this form can be obtained by setting \\
$X = (\mathcal{A}^2(\mathbb{D})\otimes \overline{\mathcal{A}^2(\mathbb{D})})^N$, $X'
=(\mathcal{A}^2(\mathbb{D})\otimes
\overline{\mathcal{A}^2(\mathbb{D})})^N$ with the natural  Hilbert space pairing, $G(f,g) =
f(\zeta)\overline{g(\eta)}(1,...,1)$,
$\Phi = (f_1(\zeta)\overline{g_1(\eta)},$ $ ...,f_N(\zeta)\overline{g_N(\eta)})$,
which gives
\begin{equation*}
 F_{G,\Phi}(f,g) = \sum_{j=1}^N \langle f, f_j \rangle \langle g_j, g \rangle.
\end{equation*}
In the special case, when $f_j = \kb_{\vs_j}$ and $g_j = \kb_{\z_j}$, for some points
$\vs_j$ and $\z_j$ in $\mathbb{D}$, we have $G(f,g) =
( f(\vs_1)\overline{g(\z_1)}, ... ,f(\vs_N)\overline{g(\z_N)} )$, so that
\begin{equation*}
 F_{G,\Phi_\mathbf{1}}(f,g) = \sum_{j=1}^N f(\vs_j)\overline{g(\z_j)}.
\end{equation*}
One more representation of this form can be obtained by setting \\
$X = (\Ec(\mathbb{D})\otimes \Ec(\mathbb{D}))^N$, $X' = (\Ec'(\mathbb{D})\otimes
\Ec'(\mathbb{C}))^N$,
\begin{eqnarray*}
 G(f,g) &=&   f(\x) \overline{g(\eta)}(1,...,1), \\
 \Phi &=& (\delta(\x - \vs_1)\delta(\eta - \z_1), ...,\delta(\x
 - \vs_N)\delta(\eta - \z_N)),
\end{eqnarray*}
which gives
\begin{equation*}
 F_{G,\Phi}(f,g) = \sum_{j=1}^N f(\vs_j)\overline{g(\z_j)}.
\end{equation*}
In either of the above cases we have
\begin{equation*}
 (Tf)(z)=\sum_{j=1}^N \,\langle f,f_j\rangle\, g_j(z) = (T_{F_{G,\Phi}} f)(z).
\end{equation*}
\end{example}

\subsection{More on distributional symbols in $\Ec'(\DD)$. Spectral properties}
In considerations above, we used the fact that a Toeplitz operator with symbol in $\Ec'(\DD)$ cannot be 'too small', unless the symbol is fairly degenerate. We will see now that such operator cannot be 'too large' either. Here the 'size' of an operator is measured by the behavior of its singular numbers.

Consider a distributional symbol of the form \eqref{Struct2}, i.e.,
\begin{equation}\label{Struct3}
    \Phi=\sum_{q}\partial^{{}^{q_1}}\dbar^{q_2}u_q,
\end{equation}
If $\Phi\in \Ec'(\DD)$, the support of the distribution $\Phi$ is contained in the disk $\DD_{r'}$ centered at the origin and having radius $r'<1$. Therefore the continuous functions $u_q$ in \eqref{Struct3} can be chosen to have their support inside a slightly larger disk, $\supp(u_q)\subset \DD_r$, with any $r\in(r',1)$.

Consider the sesquilinear form of the operator  $T_\Phi$ on the space $\mathcal{A}^2(\mathbb{D})$:

\begin{equation*}
    \langle T_\Phi f,g\rangle=(\Phi, f\bar{g}).
\end{equation*}

We substitute here the expression \eqref{Struct3} for the distribution $\Phi$, and after integration by parts, arrive at

\begin{equation}\label{SesqFormDistr2}
    \langle T_\Phi f,g\rangle=\sum_{q_1,q_2}(-1)^{q_1+q_2}\int_{\DD_r}u_q(z) \partial^{q_1} f(z) \overline{\partial^{q_2}g(z)} dV(z).
\end{equation}

Since the functions $u_q=u_{q_1,q_2}$ are continuous, they are bounded. Therefore, by the Cauchy-Schwartz inequality, \eqref{SesqFormDistr2} implies

\begin{equation}\label{SesqFormDistr3}
 |\langle T_\Phi f,g\rangle|\le C\|f\|_{H^s(\DD_r)}\|g\|_{H^s(\DD_r)},
\end{equation}
where on the right-hand side in \eqref{SesqFormDistr3} the  norms of functions $f,g$ in the Sobolev space $H^s(\DD_r)$ are present, with $s$ being the largest value of $q_1,q_2$ in the representation \eqref{Struct3}.

These Sobolev norms, as it follows  easily from the Cauchy integral formula, are majorated by the $L_2(\DD_{r_1})$-norms, for any $r_1>r$. Therefore,

\begin{equation}\label{SesqFormsEst1}
    |\langle T_\Phi f,g\rangle|\le C\|f\|_{L_2(\DD_{r_1})}\|g\|_{L_2(\DD_{r_1})}.
\end{equation}
It follows from \eqref{SesqFormsEst1} that the singular numbers of the operator $T_\Phi$ are majorated by the singular numbers of the embedding operator of $\mathcal{A}^2(\mathbb{D})$ into ${L_2(\DD_{r_1})}$, or, in other words, by the eigenvalues of the Toeplitz operator on $\mathcal{A}^2(\mathbb{D})$ with symbol being the characteristic function of the disk $\mathbb{D}_{r_1}$. The latter eigenvalues are found explicitly, and they decay as $s_n\sim r_1^{2n}$. This gives us one more  obstacle for an operator to be a Toeplitz operator with symbol in $\Ec'(\DD)$:

\begin{proposition}\label{PropEigenvalues} An operator $T$ in $\mathcal{A}^2(\mathbb{D})$, with subexponential decay of singular numbers,
\begin{equation*}
    \lim\sup\frac{|\log(s_n(T))|}{n}=0,
\end{equation*}
is \textbf{not} a Toeplitz operator with symbol in $\Ec'(\DD)$.

\end{proposition}

\section{Carleson measures for derivatives}\label{k-carl}
In order to consider as symbols the sesquilinear forms involving derivatives of all orders, we extend the notion of Carleson measures. Recall that in the theory of  Bergman type spaces, the notion of Carleson measures is used to distinguish  measures $\m$ for which the quadratic form $\int |f|^2d\m$ is majorated by the (square of) the norm of $f$ in the Bergman space in question. A detailed description of such Carleson measures can be found in \cite{Zhu} for the standard Bergman spaces in the disk and in \cite{ZhuFock} for the Fock space. As was shown in \cite{RozVas}, for the case of the Fock space, similar results can be obtained for 'Carleson measures for derivatives', or '$k$-C measures'. For such measures $\mu$, the integral $\int |f^{(k)}(z)|^2 d\mu $ can  be majorated by the square of the norm of $f$. In this section, we present a version of the 'Carleson measures for derivatives' adapted for the Bergman space $\Ac^2(\DD)$.

\begin{definition}\label{def:k-carleson}A  positive measure $\m$ on the disk $\DD$ is called a Carleson measure for derivatives of order $k$ (shortly, a $k$-C measure) if for any function $f\in\Ac^2(\mathbb{D})$,
\begin{equation}\label{CarlMeasure k}
    \left|\int_\DD |f^{(k)}|^2 d\mu\right|\le C_k(\m)\|f\|^2_{\Ac^2}.
\end{equation}
We call a \emph{complex} measure $\mu$ a $k$-C measure if its variation $|\m|$ is a $k$-C measure.
\end{definition}

For a \emph{positive}  measure $\m$, the necessary and sufficient condition for $\m$ to be a 0-C measure (or, what is the same, a C-measure) is well known since long ago, see, e.g.,  \cite{Zhu} for the history of the problem.  A number of equivalent formulations of the condition can be found in the literature. The one that is convenient for our study is given, for example, in \cite{Zhu}. It is expressed in terms of  the so called Bergman disks $B(\z,r)$, $\z\in\DD, r>0$, described, for example, in \cite[Section 4.2]{Zhu}. For our needs, a direct description of this disk is not needed; what is, actually, important is that  $B(\z,r)$ is simply the Euclidean disk $\Db(z,R)$ with center at $z=\frac{1-s^2}{1-s^2|\z|^2}\z$ and radius $R=\frac{1-|\z|^2}{1-s^2|\z|^2}s$, where $s=\tanh r\in(0,1)$.  A simple calculation shows that, for a fixed $r$ (and, therefore, $s$), the following inclusions hold
\begin{equation*}
    \Db(z,\k_1 (1-|\z|))\subset B(\z,r)\subset \Db(z,\k_2 (1-|\z|)),
\end{equation*}
with some constants $\k_1,\k_2<1$ depending on $r$ but independent of $z$. Taking this into account,  we can present the criterion given in Theorem 7.4 in \cite{Zhu} in the following form.
\begin{proposition}\label{Prop.Zhu-C} A  measure $\m$ on the disk $\DD$ is Carleson if and only~if
\begin{equation}\label{Carleson quantity}
    \varpi_0(\m)=\sup_{z\in \DD}\left\{{|\m|\left(\Db(z,\textstyle{\frac12}(1-|z|))\right)}{(1-|z|)^{-2}}\right\}<\infty,
\end{equation}
where $|\m|$ denotes the variation of the measure $\m$; moreover, the constant $C_0(\m)$ in \eqref{CarlMeasure k} is estimated from both sides by $\varpi_0(\m)$.
\end{proposition}

Note here that the coefficient $\frac12$  in \eqref{Carleson quantity} can be replaced by any number $\k\in(0,1)$, which would only has influence to the coefficients relating the above $C_0(\m)$ and $\varpi_0(\m)$.

Now we use Proposition \ref{Prop.Zhu-C} to describe $k$-C measures.

\begin{theorem}\label{Thm.k-Carleson} Let $\m$ be a locally finite measure on $\DD$. Suppose that
the quantity
\begin{equation}\label{k-Carleson.1}
    \t_k(\mu)=\sup_{z\in\DD}\left\{|\mu|\left(\Db(z,\textstyle{\frac12}(1-|z|))\right)(1-|z|)^{-2(k+1)}\right\}
\end{equation}
is finite. Then the inequality \eqref{CarlMeasure k} holds for any function $f$ in the Bergman space $\Ac^2(\DD)$ with the constant
 $$C_k(\m)=Cp^{-2k}(k!)^2\t_k(|\mu|),$$
and some constant $p\in(0,1)$.
Conversely, any positive $k$-Carleson measure $\m$, satisfying \eqref{CarlMeasure k}, must have the quantity \eqref{k-Carleson.1} finite, with the estimate $C_k(\m)\le Cp^{-2k}(k!)^2\t_k(\mu)$ holding with some constant $C$, not depending on $k$.
\end{theorem}

\begin{remark} The formulation of the direct part of the above theorem  involves the variation of the measure $\mu$. It is quite possible that for a certain non-signdefinite (or complex) measure $\mu$, the inequality
\begin{equation*}
   \int_\DD |f^{(k)}|^2 d|\mu| \le C_k(\m)\|f\|^2_{\Ac^2}
\end{equation*}
fails for some functions $f\in \Ac^2(\mathbb{D})$, while \eqref{CarlMeasure k} holds, due to some cancelation, for all $f\in\Ac^2(\mathbb{D})$ (possibly, with the integral in \eqref{CarlMeasure k}  being understood in some regularized sense).
This, in particular, means that the converse part of Theorem \ref{Thm.k-Carleson} can become wrong if the condition of the positivity is dropped.  For $k=0$ and absolutely continuous measures this cancelation effect has been discussed in \cite{TaVi1} and \cite{Zorboska03}. One can compare this effect with a similar one, discovered and investigated in \cite{MazVer}, for the inequality
\begin{equation*}
    \left|\int_{\R^d}|f(x)|^2 d\mu\right|\le C\int_{\R^d} |\nabla f|^2 dx, \quad f\in C^{\infty}_0,
\end{equation*}
with a non-signdefinite measure $\m$. Here, the authors succeeded to trace down the  cancelations in the integral and to find \emph{necessary} and sufficient conditions even for such measures. We are going to look closer at this effect on some other occasion.
\end{remark}

\begin{proof}
For a fixed $z_0\in \DD,$ we consider the (Euclidean) disk $\Db(z_0,s)$ centered at $z_0$ with radius $s<(1-|z_0|)/2$ and write the usual representation of the derivative at a point $z$, $|z-z_0|<s_1<s$ of an analytic function $f(z)$:

\begin{equation}\label{representation}
f^{(k)}(z)=k!(2\pi\imath)^{-1}\int_{|z_0-\z|=\s}(\z-z)^{-k-1}f(\z)d\z, \, s_1<\s<s.
\end{equation}
Now we fix $s_2\in(s_1,s)$ and integrate \eqref{representation}  in $\s$ variable from $s_2$ to $s$, which gives us the estimate
\begin{equation*}
|f^{(k)}(z)|\le (s-s_2)^{-1}k!(2\pi)^{-1}\int_{s_2\le |z_0-\z|\le s }|\z-z|^{-k-1}|f(\z)|dV(\z).
\end{equation*}
We choose then $s_1=\frac13 s$, $s_2=\frac23 s$.
By the triangle inequality followed by the Cauchy-Schwartz one, we obtain
\begin{gather*} |f^{(k)}(z)|\le C(s/3)^{-k-2}k!\int_{|z_0-\z|\le s} |f(\z)|dV(\z)\le \nonumber\\ C(s/3)^{-(k+1)}k!\left(\int_{|z_0-\z|\le s} |f(\z)|^2dV(\z)\right)^{\frac12},
\end{gather*}
or
\begin{equation}\label{repr2}
    |f^{(k)}(z)|^2\le  C (s/3)^{-2(k+1)}(k!)^2\left(\int_{|z_0-\z|\le s} |f(\z)|^2dV(\z)\right).
\end{equation}

The estimate \eqref{repr2} holds for all $z$, $|z-z_0|<s/3$. Therefore we can integrate it over the disk $\Db(z_0,s_1)$ with respect to the measure $\m$, which gives
\begin{gather}\label{repr3}
    \left| \int_{\Db(z_0,s_1)}|f^{(k)}(z)|^2 d\m\right|\le\\ \nonumber C \m(\Db(z_0,s_1))  (s/3)^{-2(k+1)}(k!)^2\left(\int_{|z_0-\z|\le s} |f(\z)|^2dV(\z)\right).
\end{gather}
It is known (see, e.g., \cite{Zhu}, Lemma 4.7, 4.8) that it is possible to find a locally finite covering $\Xi$ of the unit  disk $\DD$ by disks of the type $\Db(z_0,s_2)$, with $z_0\in \DD$ and $s_2=(1-|z_0|)/6$, so that the concentric disks $\Db(z_0, (1-|z_0|)/2)$ form a covering $\widetilde{\Xi}$ of $\DD$ of finite, moreover, controlled  multiplicity. The latter  means that the number
 $$m(\widetilde{\Xi})=\max_{z\in\DD}\#\{\Db\in \widetilde{\Xi}:z\in \Db\}$$
is finite and $m(\widetilde{\Xi})<C$. After adding up the inequalities of the form \eqref{repr3} over all disks $\Db=\Db(z_0,s_1)\in \Xi$, we obtain the required estimate for $\int_{\DD}|f^{k}|^2d\mu$. (Note that the number $p$ given by this calculation equals $1/9$. This value surely can be somewhat improved, but we are not doing this here.)

The converse part of the theorem is not needed in the present paper. We just mention that its proof follows almost literally the proof of Theorem 6.2.2 in \cite{Zhu}, using the same test functions.
\end{proof}

Having Theorem \ref{Thm.k-Carleson} at disposal, we introduce the classes of $k$-C measures.
\begin{definition}\label{Def.Mk} Fix a number $p\in(0,1)$. The class $\MF_{k,p}$ consists of measures $\m$ on $\DD$ such that
\begin{equation}\label{Mk}
    \sup_{z\in{\DD}}\left\{ |\m|\left(\Db(z, \textstyle{\frac12}(1-|z|))\right)(1-|z|)^{-2(k+1)}(k!)^2 p^{-2k}\right\}\equiv\varpi_{k,p}(\m)<\infty.
\end{equation}

\end{definition}

It follows from Theorem \ref{Thm.k-Carleson} that for any integer $k$, the classes $\MF_{k,p}$, for all values of $p$, coincide with the set of $k$-C measures. The quantity $\varpi_{k,p}(\m)$ differs from $\t_k(\m)$, introduced previously in \eqref{k-Carleson.1}, only by a numerical factor depending on $k$. This numeric factor is, however, important in studying relations between the $k$-C measures for different values of $k$.

  It is convenient to extend this definition to half-integer values of $k$.
\begin{definition}Let $k\in \Z_++\frac12$ be a half-integer. Fix $p\in(0,1)$. The class $\MF_{k,p}$ consists of measures $m$ on $\DD$ such that
\begin{equation}\label{Mk12}
     \sup_{z\in{\DD}}\left\{ |\m|\left(\Db(z, \textstyle{\frac12}(1-|z|))\right)(1-|z|)^{-2(k+1)}\G(k+1)^2 p^{-2k}\right\}\equiv\varpi_{k,p}(\mu)<\infty.
\end{equation}
\end{definition}
Further on, the parameter $p$ is fixed and it will be omitted in the notation, $\MF_{k,p}\equiv\MF_k$, $\varpi_{k,p}(\m)=\varpi_k(\m).$ We call $\varpi_k(\m)$ the \emph{$k$-norm} of the measure $\m$.

The following important fact is an easy consequence of \eqref{Mk}.

\begin{proposition}\label{Prop.MkMl}
Let $\m\in\MF_k$  for some integer or half-integer $k\ge0$. Then for all integer or half-integer  $l\ge0$ the measure $(1-|z|)^{2(l-k)}\m$ belongs to $\MF_l$ and
\begin{equation}\label{overnorn}
    \varpi_{l}((1-|z|)^{2(l-k)}\m)=p^{2(k-l)}(\G(l+1)/\G(k+1))^2\varpi_k(\m).
\end{equation}
In particular, for any $k$, $l$, a $k$-C measure $\m$ becomes an $l$-C measure after the multiplication by $(1-|z|)^{2l-2k}$.
\end{proposition}
We denote the coefficient in \eqref{overnorn} by $M_{l,k}$, i,e.,
\begin{equation*}
 M_{l,k} = p^{2(k-l)}(\G(l+1)/\G(k+1))^2.
\end{equation*}

Now we consider sesquilinear forms that correspond to $k$-C measures.

\begin{proposition}\label{Prop.Fmu}Let $\m$ be a  measure on $\DD$.  Given integers $l,j\ge0$, consider the sesquilinear form
\begin{equation}\label{klj-form}
    F_{\m,l,j}(f,g)=\int_{\DD}\partial^l f(z)\overline{ \partial^jg(z)}d\m(z), \, f,g\in \Ac^2.
\end{equation}
If $\m\in\MF_k$ and $2k=l+j$, then the sesquilinear form \eqref{klj-form} is bounded in $\Ac^2(\mathbb{D})$ and
\begin{equation}\label{klj-formbdd}
|F_{\m,l,j}(f,g)|\le C  M_{l,k}^{\frac12}M_{j,k}^{\frac12}\varpi_k(\m) \|f\|\|g\|,
\end{equation}
where the constant $C$ does not depend on $l,j$. The integers $l,j$ will be referred to as 'the differential order of the form \eqref{klj-form} further on.
\end{proposition}
\begin{proof}To estimate the sesquilinear form \eqref{klj-form}, we apply the Cauchy inequality, having  previously  multiplied the integrand by $(1-|z|)^{l-k}(1-|z|)^{j-k}\equiv1$. We have
\begin{gather*}\label{klj-estimate}
    |F_{\m,l,j}(f,g)|\le \left(\int_{\DD}|\partial^l f(z)|^2 (1-|z|)^{2(l-k)} d|\m|(z)\right)^{\frac12}\\ \nonumber \times \left(\int_{\DD}|\partial^j g(z)|^2 (1-|z|)^{2(j-k)} d|\m|(z)(z)\right)^{\frac12}.
\end{gather*}
As it is explained above, the measures $(1-|z|)^{2(l-k)} |\m|$ and $(1-|z|)^{2(j-k)} |\m|$ belong, respectively, to $\MF_l$ and $\MF_j$, with norms $M_{l,k}\varpi_k(\m)$ and $M_{j,k}\varpi_k(\m)$. Theorem \ref{Thm.k-Carleson}, applied to these two measures justifies \eqref{klj-formbdd}.
\end{proof}
\begin{remark}\label{remark:measureDistr} In case of the measure $\mu$ having compact support in $\DD$, one may equivalently treat the sesquilinear form $F_{\m,l,j}(f,g)$ as the form  generated by the derivative of the measure $\m$ considered as  a distribution in $\Ec'(\DD)$. In fact, by the definition of the derivative of a distribution,
\begin{equation*}
    F_{\m,l,j}(f,g)=(-1)^{l+j}\left(\partial^l\overline{\partial}^j\m, f\bar{g}\right).
\end{equation*}
However, with the condition of compact support dropped, there is no immediate relation of \eqref{klj-form} with derivatives of distributions. This observation can be compared with considerations in \cite{PeTaVi3}, where a certain class of derivatives of measures without compact support in $\DD$ condition has been studied.
\end{remark}

It is well known, see, again, \cite{Zhu2}, that the property of the operator with symbol being measure to be compact is characterized by the measure being a \emph{vanishing Carleson} one. In a similar way we introduce vanishing $k$-C measures.

\begin{definition}\label{def.vanishing}The measure $\m$ on $\DD$ is called vanishing $k$-C measure if
\begin{equation*}
    \lim_{r\to0}\sup_{|z|>r}\{|\m|\left(\Db(z,\textstyle{\frac12}(1-|z|))\right)(1-|z|)^{-2(k+1)}\}=0.
\end{equation*}
The set of vanishing $k$-C measures will be denoted by $\MF^0_k$.
\end{definition}
Compactness of the operators $T_F$ with defining form $F$ \eqref{klj-form} follows in a usual way, as soon as it is known that the measure $\m$ is a $k$-C vanishing one.
\begin{proposition}\label{prop.comp} Suppose that $\m\in \MF^0_k$. Then the operator defined by the sesquilinear form \eqref{klj-form} is compact on $\Ac^2(\mathbb{D})$.
\end{proposition}
\begin{proof} For chosen $\e>0$, by the definition of vanishing measures, we find $r<1$, so that
\begin{equation*}
    \sup_{|z|>r}\{|\m|\left(\Db(z,\textstyle{\frac12}(1-|z|))\right)(1-|z|)^{-2(k+1)}\}<\e.
\end{equation*}
With $\h$ being the characteristic function of the disk $|z|\le r$, we  set $\m_1=\h \m$ and $\m_2=\m-\m_1$.  Having compact support in $\DD$, the sesquilinear form \eqref{klj-form} with $\m$ replaced by $\m_1$ defines a compact operator. On the other hand, for the measure $\m_2$ the quantity $\varpi_k(\m_2)$ is bounded by a multiple of $\e$, therefore the norm of the corresponding operator can be made arbitrarily small.
\end{proof}

\section{Operators with highly singular symbols}\label{hyperfun}

In this section we will use the estimates, just established, to associate Toeplitz operators to highly singular symbols. Any distribution with compact support in $\DD$ has finite order and therefore it fits into the scheme of Example 3.2. In what follows we introduce symbols which (at least, formally) are not distributions any more. They are infinite sums of distributions of finite order, such that each one of then defines a bounded Toeplitz operator via corresponding sesquilinear form. If the orders of the terms in the above sum are not bounded, the resulting object is not a distribution any more, being an object  which has something common with hyper-functions. A possibility of using some hyper-functions as symbols of Toeplitz operators on the Bergman space was mentioned in \cite{PeTaVi2}.
\subsection{Operators of finite type}
Let, for a certain pairs of integers $l,j\ge0$ with $l+j=2k$, a measure $\m_{l,j}\in\MF_{k}$ is given (we admit a possibility of $\m_{l,j}=0$ for some pairs).

By Proposition \ref{Prop.Fmu}, the sesquilinear form \eqref{klj-form} with $\m=\m_{l,j}$ is bounded on $\Ac^2(\mathbb{D})$ and therefore defines a Toeplitz operator, which, following \eqref{klj-form},  we  denote by $T_{\m_{l,j}, l,j}$ (where the indices $l,j$  numerate the measures and simultaneously indicate the differential order of the sesquilinear form), or, simply, by $T_{l,j}$, keeping the dependence on the measure $\m_{l,j}$ in mind. By \eqref{klj-formbdd}, the the norm of $T_{{l,j}}$ admits the estimate
\begin{equation*}
    \|T_{{l,j}}\|\le C M_{l,k}^{1/2}M_{j,k}^{1/2}\varpi_k(\m_{l,j}).
\end{equation*}
 Thus defined operator $T_{{l,j}}$ fits  the general scheme of Section \ref{se:general}. To see this, we take $X=L_1(\DD,|\m_{l,j}|)$, $X'=L_\infty(\DD,\m)$ and $\F={\r}\in X'$, $\r$ being the Radon-Nikodym derivative of $\m_{l,j}$ with respect to $|\m_{l,j}|$. By \eqref{klj-formbdd}, the $X$-valued sesquilinear form $G$ on $\Ac^2(\mathbb{D})$ defined by
\begin{equation*}
    G(f,g)=f\bar{g},
\end{equation*}
is bounded, and the form $\F$ given by
 \begin{equation*}
    \F(f\bar{g})=\int_{\DD}f\bar{g}\r d|\m_{l,j}|=\int_{\DD}f\bar{g} d\m_{l,j},
 \end{equation*}
defines by \eqref{Def:operatorTF} the Toeplitz operator $T_{\m_{l,j},l,j}\equiv T_{{l,j}}$.

For a finite collection of measures $\pmb{\pmb{\m}}=\{\m_{l,j}\}$, we consider the sum $T_{\pmb{\pmb{\m}}}=\sum{T_{{l,j}}}$ of the above defined operators $T_{{l,j}}$. The operator $T_{\pmb{\pmb{\m}}}$ is obviously bounded. To fit it to  the scheme of Section \ref{se:general}, we take
\begin{equation*}
 X=\prod L_1(\m_{l,j}), \quad X'=\prod L_\infty(\m_{l,j}), \quad \mathrm{and} \quad \F=\{\r_{l,j}\}\in X',
\end{equation*}
where, again, $\r_{l,j}$ is the Radon-Nikodym derivative of $\m_{l,j}$ with respect to $|\m_{l,j}|$.

The corresponding operators are called operators of finite type.

\subsection{Operators of norm almost finite type.}
Now we consider infinite sums of such operators. Suppose that there are infinitely many nontrivial measures   in the set $\pmb{\pmb{\m}}$.
\begin{definition}\label{normaf}
The collection $\pmb{\pmb{\m}}$ of measures $\{\m_{l,j}\in \MF_{(l+j)/2}\}$, $l,j=0,1,\dots$, is said to be of \emph{norm almost finite type } if
\begin{equation}\label{normAFtype}
    \sum_{l,j}M_{l,\frac{l+j}{2}}^{\frac12}M_{j,\frac{l+j}2}^{\frac12}\varpi_{\frac{l+j}2}(\m_{l,j})<\infty.
\end{equation}
\end{definition}
If the collection $\pmb{\pmb{\m}}$ of measures is of norm almost finite type, then,
by Proposition \ref{Prop.Fmu}, each sesquilinear form $F_{\m_{l,j},l,j}$ as in \eqref{klj-form}, is bounded on $\Ac^2(\mathbb{D})$ and thus defines a bounded Toeplitz operator $T_{{l,j}}$, subject to the following norm estimate:
\begin{equation*}
    \|T_{{l,j}}\|\le CM_{l,\frac{l+j}{2}}^{\frac12}M_{j,\frac{l+j}2}^{\frac12}\varpi_{\frac{l+j}2}(\m_{l,j}).
\end{equation*}
By condition \eqref{normAFtype}, the series $\sum T_{{l,j}}$ is norm convergent. We will call the sum of this series \emph{the Toeplitz operator} with symbol being the collection $\pmb{\pmb{\m}}$ of measures, and denote it by $T_{\pmb{\pmb{\m}}}=\sum T_{{l,j}}$.

It follows automatically from Proposition \ref{prop.comp} that if the measures $\m_{l,j}$ are $(l+j)/2$-vanishing ones, then $\sum T_{\m_{l,j}}$ is a compact operator. In particular, the Toeplitz operator defined by a collection of measures of norm almost finite type, each of which having a compact support in $\DD$, is a compact operator.

\subsection{Examples}
\begin{example}\textsf{Discrete measures.}
Let $z_{l,j}$ be a set of points in the disk $\DD$ (not necessarily different ones). With a (double) sequence $\pmb{\pmb m}=\{m_{l,j}\}$  of complex numbers we associate the sequence of measures $\pmb{\pmb{\m}}=\{\m_{l,j}\}$, where $\m_{l,j}=m_{l,j}\d(z-z_{l,j})$.  Each measure $\m_{l,j}$, having compact support in $\DD$, is a $k$-C measure for any $k$; in particular, for $l+j=2k$, the estimate holds
\begin{equation}\label{pointmeasure}
\varpi_{k}(\m_{l,j})\le C \G(k+1)^2(1-|z_{l,j}|)^{-2(k+1)}p^{-(l+j)}|m_{l,j}|
\end{equation}
for some value of $p<1$, with some constant $C$ not depending on $k$. This means that the norm of the Toeplitz operator $T_{{l,j}}$ defined by the above measures via the sesquilinear form \eqref{klj-form} of the differential order $(l,j)$,  with $\m=\m_{l,j}$,  does not exceed the quantity in \eqref{pointmeasure}. By Remark \ref{remark:measureDistr}, the operator $T_{{l,j}}$ can be equivalently considered as the Toeplitz operator with symbol-distribution $\Phi_{l,j}=(-1)^{l+j}m_{l,j}\overline{\partial}^j\partial^l\d(z-z_{l,j})$.

Now suppose that the numbers $m_{l,j}$ decay so rapidly that
\begin{equation*}
    \sum_{l,j}\varpi_{k}(\m_{l,j})M_{l,k}^{\frac12}M_{j,k}^{\frac12}<\infty.
\end{equation*}
Recalling the definition of $M_{l,k}$ and $\varpi_{k}(\m_{l,j})$ in \eqref{pointmeasure}, we can rewrite the above condition as
\begin{equation*}
    \sum_{l,j}\G(l+1)\G(j+1)(1-|z_{l,j}|)^{-(l+j+2)} p^{-(l+j)}|m_{l,j}|<\infty.
\end{equation*}
 In this case the sum of the norms of operators $T_{{l,j}}$, defined above, is finite and therefore the Toeplitz operator $T_{\pmb{\pmb{\m}}}$ is defined as one corresponding to a collection of measures of norm almost finite type. Since all measures in $\pmb{\pmb{\m}}$ have compact support in $\DD$, the operator $T_{\pmb{\pmb{\m}}}$ is compact. It is important to write explicitly the sesquilinear form generating the operator $T_{\pmb{\pmb{\m}}}$:
\begin{equation}\label{DiscrFormOper}
    F_{\pmb{\pmb{\m}}}(f,g)=\sum_{l,j}m_{l,j}f^{(l)}(z_{l,j})\overline{g^{(j)}(z_{l,j})},
\end{equation}
where the sum converges absolutely.
\end{example}
\begin{example}\label{centerpoint}\textsf{Measures with support at the origin.}
A special, rather interesting, case in the previous example corresponds to the case of a collection of measures $\m_{j,l}$ being the point masses at the origin, $\m_{l,j}=m_{l,j}\d(z)$. Each of these measures is a $k$-C measure with any $k$. Here, by using more exact estimates then  the ones in \eqref{klj-formbdd}, one can improve the results for the sesquilinear form \eqref{DiscrFormOper}  in comparison with the general discrete measure case.
The expression \eqref{DiscrFormOper} for the sum of sesquilinear forms \eqref{klj-form} (with $\m=\m_{l,j}$) generated by the sequence of discrete measures  becomes now
\begin{equation}\label{DiscrFormOperCenter}
    F_{\pmb{\pmb{\m}}}(f,g)=\sum_{l,j}m_{l,j}f^{(l)}(0)\overline{g^{j}(0)}.
\end{equation}
To estimate a single term in \eqref{DiscrFormOper},
we use, again, the Cauchy formula for the derivative of an analytic function  $f\in\Ac^2(\mathbb{D})$:
\begin{equation}\label{Center1}
    f^{(j)}(0)=(2\pi\imath)^{-1}j!\int_{|\z|=r}f(\z)\z^{-(j+1)}d\z,
\end{equation}
where $r\in(0,1)$. By the Cauchy inequality,\eqref{Center1} implies
\begin{equation}\label{Center2}
    |f^{(j)}(0)|^2\le C j!^2 r^{-2j-1}\left|\int_{|\z|=r}|f(\z)|^2|d\z|\right|
\end{equation}
We multiply \eqref{Center2} by $r^{2j+1}$ and integrate in $r$ from $0 $ to $1.$ As a result we obtain
\begin{equation}\label{Center3}
    |f^{(j)}(0)|^2\le C j!^2 (j+1)\int_{\DD}|f(\z)|^2 dV=Cj!^2(j+1)\|f\|^2,\end{equation}
with some constant not depending on $j$. Now, we are able to estimate  the sesquilinear form \eqref{DiscrFormOperCenter}. By \eqref{Center3}, we have
\begin{gather*}
     |F_{\pmb{\pmb{\m}}}(f,g)|\le \sum_{l,j}|m_{l,j}||f^{(l)}(0)||g^{(j)}(0)|\le\\ \nonumber
    C\sum_{l,j}|m_{l,j}|l!j!(l+1)^{\frac12}(j+1)^{\frac12}\|f\|\|g\|.
\end{gather*}
Thus, if
\begin{equation*}
    \sum_{l,j}|m_{l,j}|l!j!(l+1)^{\frac12}(j+1)^{\frac12}<\infty,
\end{equation*}
 the system $\pmb{\pmb{\m}}$ of measures is of norm almost finite type and the sesquilinear form $F_{\pmb{\pmb{\m}}}$ is bounded in $\Ac^2(\DD)$.

 A less general class of distributions of infinite order supported at the origin was considered in \cite[Example 1 in Sect.2]{PeTaVi2}, where  distributions of the form $\sum m_{\a} D^{\a}\d(z)$ were treated as hyper-functions.
\end{example}

\subsection{Radial measures.} In Section 4.2 we considered radial operators with symbol being  bounded radial functions. Now we are ready to study radial sesquilinear forms  containing  derivatives of unbounded order as symbols.

Let ${\pmb{\pmb{\m}}}=\{\m_{q,i,q'}\}$, $q,i,q'\in \Z_+$, be a collection of measures on $\DD$, with $\m_{q,i,q'}$ being an $(q+i+q')/2$-C measure. We suppose that each of the measures in ${\pmb{\pmb{\m}}}$ is radial. Similar to the case of functions, this means that each $\m_{q,i,q'}$ is invariant with respect to the rotation $U_\th:E\mapsto e^{\imath \th}E$, i.e., $\m_{q,i,q'}(E)=\m_{q,i,q'}(U_\th E)$ for all Borel  sets $E$.

Denote by $\pmb{\th}$ the weighted circular  derivative, $\pmb{\th} f(re^{\imath\th})=\imath r^{-1}\frac{\partial}{\partial \th}f(re^{\imath\th})$, and by $\pmb{\r}$ the radial derivative.

 Any differential operator on $\DD$ commuting with $U_\th$ has the form $\sum b_{q,i}(r)\pmb{\th}^i\pmb{\r}^q$, where the coefficients $b_{q,i}$ depend only on $r$. Therefore, any sesquilinear form \eqref{klj-form} associated with a circular measure $\m_{l,j}$ can be, by passing to polar co-ordinates and rearranging terms, transformed to
\begin{equation}\label{radForm}
    F_{\m_{l,j},l,j}(f,g)=\sum_{q'+q+i=l+j}\int_{\DD}\pmb{\r}^i\pmb{\th}^q f \overline{\pmb{\r}^{q'}g}d\m_{q,i,q'}, \quad f,g\in\Ac^2(\DD),
\end{equation}
with some $(q+i+q')/2$- C- measures  $\m_{q,i,q'}$.

Note that although the functions $f,g$ seem to appear asymmetrically in \eqref{radForm}, i.e., $f$ enters with the circular derivative $\pmb{\th}$, while $g$ enters without this derivative, this asymmetry is only superficial since the measure $\m_{q,i,q'}$ is rotation invariant and therefore all circular derivatives can be moved from $g$ to $f$ by means of the partial integration in $\th$ variable.

 Next, in the study of the Toeplitz operator corresponding to the sum of forms \eqref{radForm}, we suppose, for the sake of brevity, that all measures $\m_{q,i,q'}$ are zero inside the disk with radius $\frac12$, in order to be able to focus on the boundary behavior.

In the integrand in \eqref{radForm}, we have a derivative of $f$ of order $q+i$ and the derivative of $g$ of order $q'$.
By Proposition \ref{Prop.Fmu}, under the condition $\m\in \MF_k$, $2k=q+i+q'$, the sesquilinear form \eqref{radForm} is bounded on $\Ac^2(\mathbb{D})$, and for the norm of the corresponding operator $T_{{q,i,q'}}$ we have the estimate
\begin{equation*}
    \|T_{q,i,q'}\|\le C M^{\frac12}_{q+i,k}M^{\frac12}_{q',k}\varpi_k(\m_{q,i,q'}).
\end{equation*}
Thus, if the norms of the measures $\m_{q,i,q'}$ decay sufficiently rapidly, so that $\sum M^{\frac12}_{q+i,k}M^{\frac12}{q',k}\varpi_k(\m_{q,i,q'})<\infty,$ the sum of the  forms \eqref{radForm} is of norm almost finite type and therefore defines a bounded Toeplitz operator.

\begin{example}\textsf{Measures supported on concentric circles.}
Denote by $\T_k$ the circle centered at the origin and with  radius $r_k\in (0,1)$, where $r_k$ is a monotonically growing  sequence, $r_k\to 1.$ For a sequence of complex numbers $m_k$, we introduce for any $k$ a measure $\m_k$ being  $m_k$ times the normalized Lebesgue measure on the circle $\T_k$:
\begin{equation}\label{measureCircle}
\m_k=m_k\delta(r-r_k)\otimes (2\pi)^{-1}d\th.
\end{equation}

Following our general construction, we consider the sesquilinear forms
\begin{equation}\label{radForm3}
 G_{k}(f,g)=\int_{\T_k}\pmb{\r}^{q_k}\pmb{\th}^{i_k}f\overline{\pmb{\r}^{q'_k}g}d\m_k,
\end{equation}
where $i_k,q_k,q'_k$ are some non-negative integers, $i_k+q_k+q'_k=2k$, and, recall, that $\pmb{\r}$, $\pmb{\th}$ denote, respectively, the radial and circular derivatives. The sesquilinear form \eqref{radForm3}, being a one corresponding to a measure with compact support in $\DD$,  is bounded on $\Ac^2(\mathbb{D})$.  The norm of  $\m_k$ can be easily estimated using its definition \eqref{Mk12},
\begin{equation}\label{Mkrad}
    \varpi_{k,p}{\m_k}\le C |m_k|\G(k+1)^2 p^{-2k}(1-|r_k|)^{-2k-1}.
\end{equation}
From \eqref{Mkrad} we obtain the estimate for the norm of the operator $T_{G_k}$:
\begin{equation}\label{Mkrad2}
    \|T_{G_k}\|\le C |m_k|\G(i_k+q_k)\G(q'_k)p^{-2k}(1-|r_k|)^{-2k-1}.
\end{equation}
Therefore, if the coefficients $m_k$ decay sufficiently rapidly, so that the series of terms \eqref{Mkrad2} converges, then the system of measures $\m_k$ is of norm almost finite type and defines a bounded compact Toeplitz operator.

It is interesting to present an explicit formula for the action of an operator with a circular symbol supported on a circle, in other words,  generated by the sesquilinear form $F$ given by \eqref{radForm3}, with measure \eqref{measureCircle}.

Since for the analytic function $f(z)$, we have $\pmb{\r}f(w)=\frac{w}{|w|}\partial f(w), \pmb{\th}=i\frac{w}{|w|}\partial f(w)$, the basic formula \eqref{Def:operatorTF} gives
\begin{gather*}
    (T_F f)(z)=F(f,\kb_z)=\\ \nonumber r_k^{-1}\int_{w=r_k^{i\th}}
    \imath^{i_k}(\frac{w}{|w|}\partial_w)^{q_k+i_k}f(w)
    (\frac{w}{|w|}\partial_w)^{q_k'}(z-w)^{-2}d\th.
\end{gather*}

\end{example}
\subsection{Symbols of weak almost finite type.}
If one does not possess an estimate for the norms of the terms of the measures in $\pmb{\pmb{\m}}$, so that  \eqref{normAFtype} is not available, it may happen that it is still possible  to associate a bounded Toeplitz operator with $\pmb{\pmb{\m}}$.
\begin{definition}\label{Def.weak} Let $F(f,g)$ be a bounded sesquilinear form on the Bergman space $\Ac^2(\DD)$. We say that this form is a symbol of \emph{weak almost finite type} if there exists a collection
 $$\pmb{\pmb{\m}}=\{\m_{l,j}\}_{l,j \in \mathbb{Z}_+}$$
  of $(l+j)/2$-C-measures such that, for each $f,g\in\Ac^2(\DD) $ the series
\begin{equation}\label{eq:almost finite.1}
\sum_k\sum_{\max(l,j)=k}F_{l,j}(f,g) = \sum_k\sum_{\max(l,j)=k}
\int_{\DD}\partial^lf\overline{\partial^jg}d\mu
\end{equation}
converges to $F(f,g)$.
\end{definition}
\begin{remark}We do not require that the series in \eqref{eq:almost finite.1} converges absolutely.
\end{remark}
The Banach-Steinhaus theorem (more exactly, its version for sesquilinear forms, see, e.g., \cite[Sect. 7.7]{Edv}) implies that the condition of boundedness of the sesquilinear form
$F(f,g)$ in this definition is superfluous, it
follows automatically from \eqref{eq:almost finite.1}.

The condition \eqref{eq:almost finite.1} can be in an obvious way formulated in another form,
involving Toeplitz operators corresponding to
$\sum_k\sum_{\max(l,j)=k} F_{l,j}$.

\begin{proposition}\label{prop:weak conv.1}Let $F(f,g)$ be a symbol of weak almost-finite type.
Then the sequence of Toeplitz operators $T_s=\sum_{k\le s}T_{F_k}$
converges weakly to the Toeplitz operator $T_F$ which corresponds to the bounded sesquilinear form $F$.
\end{proposition}

It follows, in particular, that for symbols of weak almost-finite type, the point-wise
convergence takes place:
\begin{proposition}\label{prop:pointwise1} If $F$ is a symbol of weak almost-finite type, then,
for any $f\in \mathcal{A}^2(\DD)$ and any $z\in\DD$, the
sequence $(T_sf)(z)$ converges to $(T_{F}f)(z)$.
\end{proposition}
The proof follows immediately from the relation $(T_F f)(z)=F(f,\kb_z)$ and similar relations for $F_k$.

Moreover, we have the following partially converse statement.

\begin{proposition} If the norms of the operators $T_s$ are uniformly boun\-ded and the sequence $T_s f$
converges to $T_Ff$ point-wise, i.e., $(T_sf)(z)$ converges to $(T_{F}f)(z)$ for all $f\in \Ac^2(\DD)$ and
$z\in \DD$, then $T_s $ converges to $T$ weakly
\end{proposition}
The (rather standard) proof goes the following way. A given $g\in \Ac^2(\DD)$ satisfies $g(z)=\int  \kb_z(w) g(w) dV(w)$, therefore $g$ can be norm approximated, with norm error less than $\varepsilon$, by a finite linear combination
$g_\varepsilon$ of $\kb_{z_m}$. On such combinations the convergence of $(T_s f,g_\varepsilon)$  follows by linearity, and the possibility of passing to the limit follows from the uniform  boundedness.

As the following theorem shows, the notion we have just introduced might be considered too general.
\begin{theorem}\label{th:weak-universal} Any bounded operator in $\Ac^2({\DD})$ is
an operator with symbol of weak almost finite
type. Moreover, the collection of measures $\pmb{\pmb{\m}}$ can be chosen so that $\m_{l,j}$ are point masses at the origin.
\end{theorem}
\begin{proof} Let $T$ be a bounded operator  in $\Ac^2(\DD)$. We consider the
representation of the operator $T$ as an infinite matrix in
the orthogonal basis $e_j$ defined in \eqref{basis}:
\begin{equation*}
    T=\sum_{l,j} T_{l,j}=\sum_{l,j}P_l TP_j,
\end{equation*}
where $P_j=P_{jj}$ is the orthogonal projection onto the one-dimensional subspace spanned by the function
$e_j$ (see \eqref{basis}). As described in Example 3.3, each
operator $T_{lj}=P_l TP_j$, being a rank one operator, with the source
 subspace spanned by $e_j$ and the range spanned by $e_l$, is in
fact the operator $\s_{lj} P_{lj}$, with  numerical coefficient $\s_{lj}=(Te_l,e_j)$
and the rank one operators $P_{lj}$ described in
\eqref{TFEx2}. Now, by \eqref{TFEx1}, \eqref{TFEx2}, the operator $\s_{lj} P_{lj}$ is
the Toeplitz operator associated with the
distributional symbol $\s_{lj}\F_{lj}\in \Ec'(\DD)$. Moreover, any  such symbol is, by
\eqref{Struct2}, a collection of derivatives of continuous functions with compact support, and
the latter correspond to $k$-C measures for any $k$. Therefore, any symbol
$\s_{lj}\F_{lj}$ is the sum of derivatives of order not greater than $l+j$
of $(l+j)/2$-C measures. Finally, we establish the convergence of the sesquilinear forms, required
by the theorem. Denote by $\Pc_m$ the projection
  $\Pc_m=
 \sum_{l\le m}P_l$. Thus,
 the operator
\begin{equation*}
    \Pc_m T \Pc_m=\sum_{l,j\le m}T_{lj}
\end{equation*}
is the Toeplitz operator with symbol being the sum of derivatives of \\ $m$-C measures. On the other hand, the projections $\Pc_m$ converge
strongly to the identity operator as $m\to\infty$, therefore for any $f,g\in
\Ac^2(\mathbb{\DD})$, we have
\begin{equation*}
    \langle\Pc_m T\Pc_m f,g\rangle=\langle T\Pc_m f,\Pc_m g\rangle\to \langle T f, g\rangle.
\end{equation*}
\end{proof}

Consider now the special case of the above theorem when the operator $T$ is compact. Then the
sequence $\{\Pc_m T\Pc_m\}_{m \in \mathbb{Z}_+}$ converges to $T$ in norm,
\begin{equation*}
 \lim_{m \to \infty} \Pc_m T\Pc_m = T.
\end{equation*}
The sesquilinear form of the operator $T$ is the (norm) limit of the sequence of forms of
Toeplitz operators $\Pc_m T\Pc_m$, each one of which, as it was stated in the proof of Theorem
\ref{th:weak-universal}, is a sum of derivatives of order not greater than $l+j$
of $(l+j)/2$-C measures. Note that the differential order of these forms may tend to infinity as $m
\to \infty$. Thus the limiting form will have, in general, an infinite differential order. We
will identify the sesquilinear form of the operator $T$ with the sequence
$\pmb{\m}=\{\m_{k}\}_{k \in \mathbb{Z}_+}$, where each $\m_{k}$ is a form of the finite
rank Toeplitz operator $\Pc_k T\Pc_k$. Thus we come to the following important result.

\begin{theorem}
 Each compact operator on $\Ac^2(\mathbb{\DD})$ is a Toeplitz operator defined, in
 general, by a sesquilinear form of norm almost finite type  corresponding to a collection of measures being point masses at the origin.
\end{theorem}



\begin{thebibliography}{2}
\bibitem{AlexRoz} A.~Alexandrov, G.~Rozenblum,
\emph{Finite rank Toeplitz operators: some extensions of D.Luecking's theorem}, J. Funct. Anal. \textbf{256} (2009)  2291--2303.



\bibitem{Edv} R.E.~Edwards. Functional Analysis, Theory and Applications.  Holt, Rinehard and
    Winston, NY, 1965


\bibitem{EsmMax} K.~Esmeral, E.~Maximenko,  $C^*$-\emph{algebra of angular Toeplitz operators on Bergman spaces over the upper half-plane}. Commun. Math. Anal. 17 (2014), no. 2, 151-162.
\bibitem{EsmMaxVas}
K.~Esmeral, E.~Maximenko, and N.~Vasilevskii. \emph{$C^*$-algebra generated by angular Toeplitz operators on the weighted Bergman spaces over the upper half-plane}. Preprint 2014, 15 p.
\bibitem{Gelf}I.M.~Gelfand , G.E.~Shilov,  Generalized functions, II. Spaces of fundamental and generalized functions (AP, 1968)
\bibitem{GrMaxVas}
S.~Grudsky, E.~Maximenko, and N.~Vasilevski. \emph{Radial Toeplitz Operators on the Unit Ball and Slowly Oscillating Sequences}. Commun. Math. Anal., v. 14 (2013), no. 2,  77-94.

\bibitem{HeKoZ} H.~Hedenmalm, B. Korenblum, and K. Zhu, Theory of Bergman spaces. Graduate Texts in Mathematics, 199. Springer-Verlag, New York, 2000. x+286 pp.
\bibitem{Ja1} J.~Janas, \emph{Unbounded Toeplitz operators in the Segal--Bargmann spece}, Studia Math. \textbf{99} (1991) 87--99.

\bibitem{HerreraYanezMaximenkoVasilevski2013}
C.~Herrera Ya\~{n}ez, E.~Maximenko, N.~Vasilevski,
\textit{Vertical Toeplitz operators on the upper half-plane and very slowly oscillating functions.}
Integr. Equ. Oper. Theory \textbf{77}, 2013, 149--166.



\bibitem{Lue2} D.~Luecking, \emph{Finite rank Toeplitz operators in the Bergman space}. Proc. Amer. Math. Soc., 136 (2008), 1717--1723.


\bibitem{LuTa}W.~Lusky, J.~Taskinen, \emph{Toeplitz operators on Bergman spaces and Hardy multipliers.} Studia Math. 204 (2011), no. 2, 137-154.

\bibitem{MazVer}V.~ Maz'ya, I.~Verbitsky,  \emph{The Schrodinger operator on the energy space: boundedness and compactness criteria.} Acta Math. 188 (2002), 263-302.

\bibitem{PeTaVi2}A.~Per\"al\"a, J.~Taskinen, and J.~Virtanen, \emph{New results and open problems on Toeplitz operators in Bergman spaces,} New York Journal of Mathematics
\textbf{ 17a} (2011) 147-164.

\bibitem{PeTaVi3}A.~Per\"al\"a, J.~Taskinen, and J.~Virtanen, \emph{Toeplitz operators with distributional symbols on Bergman spaces.} Proc. Edinb. Math. Soc. (2) 54 (2011), no. 2, 505-514.


\bibitem{RozVas} G.~Rozenblum, N.~Vasilevski,\emph{Toeplitz operators defined by sesquilinear forms: Fock space case}. Journal of Functional Analysis, \textbf{267} (2014), 4399--4430.

\bibitem{RozToepl} G.~Rozenblum, \emph{Finite rank Toeplitz operators in the Bergman space} in: Around the research of Vladimir Maz'ya, III, Int. Math. Series (NY), 2010, 331-358.

\bibitem{SanVas}
A.~S\'anchez-Nungaray, N.~Vasilevski, \emph{Toeplitz Operators on the Bergman Spaces with Pseudodifferential Defining Symbols}. Operator Theory: Advances and Applications, v. 228 (2013), 355-374.

\bibitem{Schmidt_1924}
R.~Schmidt, \emph{\"{U}ber divergente {F}olgen and lineare {M}ittelbildungen.}
Math. Z., \textbf{22}, (1924), 89--152.

\bibitem{TaVi1} J.~Taskinen, J.~Virtanen,  \emph{Toeplitz operators on Bergman spaces with locally integrable symbols.} Rev. Mat. Iberoam. \textbf{26} (2010), no. 2, 693--706.

\bibitem{Vasil}  N.~Vasilevski, Commutative Algebras of Toeplitz Operators on the Bergman Space. Birkh\"auser, 2008.


\bibitem{Zhou_Chen_Dong_2011}
Ze-Hua Zhou, Wei-Li Chen, and Xing-Tang Dong.
\emph{The {B}erezin transform and radial operators on the {B}ergman space
  of the unit ball}.
Complex Analysis and Operator Theory, v. 7 (2013), no. 1, 313-329.

\bibitem{Zhu} K.~Zhu,
 Operator theory in function spaces. Second edition. Mathematical Surveys and Monographs, 138. American Mathematical Society, Providence, RI, 2007.

\bibitem{Zhu2} K. Zhu,  Spaces of holomorphic functions in the unit ball. Graduate Texts in Mathematics, 226. Springer-Verlag, New York, 2005.

\bibitem{ZhuFock} K. Zhu, Analysis on Fock Spaces, Graduate Texts in Mathematics, 263. Springer-Verlag, New York, 2012.

\bibitem{Zorboska03} N.~Zorboska, \emph{Toeplitz operators with BMO symbols and the Berezin transform.} Int. J.  Math. Sci. 2003, no. 46, 2929-2945.


\end{thebibliography}
\end{document}